\RequirePackage{cmap} 
\documentclass[12pt,reqno]{amsart}
\usepackage[utf8]{inputenc}
\usepackage[T2A]{fontenc}
\usepackage{amsmath,amssymb,amsthm,blkarray,graphicx,mathtools,footmisc,bbm}
\usepackage[a4paper,hmargin=2.5cm,vmargin=2.5cm]{geometry}
\usepackage[english]{babel}
\usepackage[mathscr]{euscript}
\usepackage{tikz}
\usepackage{tikz-cd}
\usetikzlibrary{braids}
\usepackage{leftindex}
\usetikzlibrary{arrows.meta,decorations.markings}

\definecolor{darkspringgreen}{rgb}{0.09, 0.45, 0.27}

\usepackage[pdftex,bookmarks=false,colorlinks=black,citecolor=darkspringgreen,debug=true,
  naturalnames=true,pdfnewwindow=true]{hyperref}

\newtheorem{thm}{Theorem}[section]

\newtheorem{cor}[thm]{Corollary}

\newtheorem{lem}[thm]{Lemma}
\newtheorem{prop}[thm]{Proposition}

\newtheorem{ex}[thm]{Example}
\newtheorem{prob}[thm]{Problem}

\theoremstyle{definition}

\theoremstyle{remark}
\newtheorem{rem}[thm]{Remark}
\newcommand{\thmref}[1]{Theorem~\ref{#1}}

\newcommand{\lemref}[1]{Lemma~\ref{#1}}

\numberwithin{equation}{section}

\newcommand{\nc}{\newcommand}
\nc{\on}{\operatorname} \nc{\wh}{\widehat}
\nc\ol{\overline} \nc\ul{\underline} \nc\wt{\widetilde}
\nc{\Hom}{\operatorname{Hom}} \nc{\reg}{\operatorname{reg}} \nc{\diag}{\operatorname{diag}}
\nc{\red}[1]{{\color{red}#1}}

\newcommand{\CC}{\mathbb{C}}
\newcommand{\tensor}{\otimes}
\newcommand{\injto}{\hookrightarrow}

\DeclareMathOperator{\End}{End}

\DeclareMathOperator{\Rep}{\underline{\smash{{\rm Rep}}}}

\tikzset{middlearrow/.style={
        decoration={markings,
            mark= at position 0.55 with {\arrow{#1}},
        },
        postaction={decorate}
    }
}

\newcommand{\bbone}{\text{\usefont{U}{bbold}{m}{n}1}}
\MakeRobust{\bbone}

\newcommand{\F}{{}_2F_1}

\begin{document}
\author{ P. E\MakeLowercase{TINGOF} \textsuperscript{$\star$}, I. M\MakeLowercase{OTORIN}\textsuperscript{$*$}, A. V\MakeLowercase{ARCHENKO}\textsuperscript{$\diamondsuit$}, \MakeLowercase{and} 
I. Z\MakeLowercase{HU}\textsuperscript{$\dagger$}}
\address{\parbox{\linewidth}{\textsuperscript{$\diamondsuit$} Department of Mathematics, University of North Carolina, Chapel Hill, NC 27599 – 3250, USA\\
\textsuperscript{$\star, *, \dagger$} Department of Mathematics, MIT, Cambridge, MA 02139, USA}}
\email{\parbox{\linewidth}{\textsuperscript{$\star$} etingof@math.mit.edu,\\
\textsuperscript{$*$} ivanm597@mit.edu,\\
\textsuperscript{$\diamondsuit$} anv@email.unc.edu,\\
\textsuperscript{$\dagger$} isaaczhu@mit.edu.}}
\title[Knizhnik-Zamolodchikov equations in Deligne categories]{Knizhnik-Zamolodchikov equations in Deligne categories}
      
\begin{abstract}
  We consider the Knizhnik-Zamolodchikov equations in Deligne Categories in the context of $(\mathfrak{gl}_m,\mathfrak{gl}_{n})$ and $(\mathfrak{so}_m,\mathfrak{so}_{2n})$ dualities. We derive integral formulas for the solutions in the first case
  and compute monodromy   
  in both cases.
\end{abstract}

\maketitle
\tableofcontents

\section{Introduction}
%\begin{itemize}
%     \item Background
%     \item paper structure, summarize the strategy (take KZ equations for $\mathfrak{gl}_n$ for (sufficiently large) $n \in \mathbb{Z}$, use a duality to reframe as dynamical equations over $\mathfrak{gl}_m$, write a solution which is analytic in $n$, interpolate this solution)
%\end{itemize}
The Knizhnik-Zamolodchikov (KZ) connection is an important object in representation theory of affine Lie algebras and quantum groups. Namely, for an arbitrary simple Lie algebra $\mathfrak{g}$ we may consider a connection $\nabla_{KZ}$ on a base space
\begin{equation}\label{base}
    \CC^r \setminus \bigcup_{1\le i<j\le r} \{z \in \CC^r | z_i-z_j=0\}
\end{equation}
given by
\begin{equation}
    \nabla_{KZ}=d-\hbar \sum_{1\le i<j \le r} \frac{d(z_i-z_j)}{z_i-z_j}\Omega_{ij},
\end{equation}
where $\hbar\in \CC$ and $\Omega_{ij}\in U(\mathfrak{g})^{\otimes r}$ is equal to
\begin{equation}
    \Omega_{ij}=\sum_{1\le a\le \dim(\mathfrak{g})}1^{(1)}\otimes \dots \otimes 1^{(i-1)}\otimes e_a^{(i)} \otimes 1^{(i+1)} \otimes \dots \otimes 1^{(j-1)}\otimes e^{a(j)} \otimes 1^{(j+1)} \otimes \dots \otimes 1^{(r)}
\end{equation}
and $e_a,e^a$ are dual bases in $\mathfrak{g}$. The Knizhnik-Zamolodchikov connection admits an obvious generalization to the case of general linear Lie algebras $\mathfrak{gl}_n$.\\

Instead of working with $U(\mathfrak{g})^{\otimes r}$ as a fiber of the trivial vector bundle with the base \eqref{base}, we choose to work with $\mathfrak{g}$-modules $V_1,\dots, V_r$. We may choose a Cartan subalgebra $\mathfrak{h}$ for $\mathfrak{g}$. Then, since the action of $\mathfrak{g}$ on $V_1\otimes \dots \otimes V_r$ commutes with $\Omega_{ij}$ and since the operators $\Omega_{ij}$ have weight $0$ with respect to $\mathfrak{g}$, it makes sense to restrict the connection to a weight space
\begin{equation}\label{KZ-weight-space}
    (V_1\otimes \dots \otimes V_r) [\mu] \subset V_1\otimes \dots \otimes V_r, \quad \mu \in \mathfrak{h}^*.
\end{equation}
If we slightly deform the KZ connection and look for the flat sections of $\nabla_{KZ}$ on \eqref{KZ-weight-space}, one can write a system of compatible dynamical equations and integral solutions which satisfy both the KZ and the dynamical equations (see \cite{FMTV}[Theorems 3.1, 3.2]).\\

The Deligne categories $\Rep(GL_t),\Rep(O_t)$ for a parameter $t\in \mathbb{C}$ are certain interpolations of the representation categories of the classical algebraic groups $GL_n$ and $O_n$ respectively (\cite{CW,Etc}). It is possible to produce a pencil of KZ connections in the setting of Deligne categories depending on the parameter $t$ (\cite{EV}[Section 5.1.4]). Therefore, we may look for the flat sections of the KZ equations in this case as well. A direct application of the approach in \cite{FMTV} fails since there are no weight spaces \eqref{KZ-weight-space} in Deligne categories. Nevertheless, it is possible to find a certain $(\mathfrak{gl}_m,\mathfrak{gl}_n)$-duality which allows us to write integral formulas for the solutions to the KZ equations for all noninteger and large enough integer $t$.\\

In the case of $\Rep(O_t)$, we can produce a similar $(\mathfrak{so}_m,\mathfrak{so}_{2n})$-duality. It gives us a new flat connection \eqref{new_connection} which can be obtained by a suitable reduction of the boundary Casimir connection in \cite{BK}[Formula (10.9)]. However, finding flat sections of the orthogonal KZ for Deligne categories or the new connection is an open problem.\\

Conjecturally (see \cite{EV}), all periodic pencils should admit integral formulas for solutions. Indeed, since they are quasi-geometric for rational values of parameters, Simpson conjectures (\cite{Sim92}) imply that they are actually geometric. Also, they have nilpotent $p$-curvature at rational parameters, so Andr\'e-Bombieri-Dwork conjectures (\cite{And89,Bom81,Dw94}) would imply that the periodic pencils are geometric. Although, there are examples where such formulas aren't known (e.g. Dunkl connections for exceptional groups). In particular, we check that for KZ equations in Deligne categories the Simpson and Andr\'e-Bombieri-Dwork conjectures indeed hold.\\
%The setting of Deligne categories is an important generalization of KZ equations in view of the recent Conjecture 4.26 in \cite{EV} stating that connections arising from braided fusion categories are always quasi-geometric. In particular, it follows from Section 6 that the KZ connection for Deligne categories gives rise to such an example.\\

The Drinfeld-Kohno theorem (e.g. see \cite{Dr1,Dr2,EK}) states that one can compute the monodromy of the KZ equations in terms of quantum $R$-matrices. We show that this result can be generalized in the context of Deligne categories and their quantizations known as skein categories. \\

A starting point of this project was an observation in \cite{TVqKZ} that the asymptotics of hypergeometric integrals, as their dimension tends to infinity, can be analyzed in some examples using the steepest descent method applied to other hypergeometric integrals of fixed dimension.\\

The paper is structured as follows. Section 2 contains preliminaries. In Section 3 we produce the $GL_t$ duality under which the KZ connection maps to the dynamical connection on a certain simple module for the dual general linear Lie algebra. The integral formulas for solutions to the dynamical equations from Section 3 are presented in Section 4. Section 5 extends the results of Section 3 regarding duality for the orthogonal case $O_t$. Finally, Section 6 generalizes the Drinfeld-Kohno theorem (\cite{Dr1,Dr2,EK}) in the context of Deligne categories.
\subsection{Acknowledgements.} Pavel Etingof's work was partially supported by the NSF grant DMS-2001318. Alexander Varchenko's work was partially supported by the Simons Foundation grant TSM-00012774. We thank UROP and UROP$+$ organizers for providing framework for this research. %We thank Alexander Varchenko for many useful discussions.
\section{Preliminaries}
\subsection{Kac-Moody Lie algebras.} Suppose we are given an integer square matrix $A$ of size $n$ and rank $l$, such that
\begin{equation}
    a_{ii}=2, \quad a_{ij}\le 0 \text{ if } i\neq j, \quad a_{ij}=0 \Rightarrow a_{ji}=0.
\end{equation}
It is called a generalized Cartan matrix. Let $\mathfrak{h}$ be a vector space of dimension $2n-l$ with independent simple co-roots $\Pi^{\vee}=\{h_1,\dots, h_n\}$ in $\mathfrak{h}$ and let $\Pi$ be a set of independent simple roots $\{\alpha_1, \dots, \alpha_n\}$ in $\mathfrak{h}^*$, such that
\begin{equation}
    \langle h_i,\alpha_j \rangle =a_{ij}.
\end{equation}
Then there exists a Lie algebra $\mathfrak{g}(A)=\mathfrak{n}_- \oplus \mathfrak{h} \oplus \mathfrak{n}_+$, such that $\mathfrak{n}_+$ is generated by elements $e_1,\dots, e_n$ and $\mathfrak{n}_-$ is generated by elements $f_1,\dots, f_n$ with relations
\begin{equation}
    [e_i,f_j]=\delta_{ij} h_i, \quad [h,h']=0, \quad [h,e_i]=\alpha_i(h)e_i, \quad [h,f_i]=-\alpha_i(h)f_i
\end{equation}
for $h,h'\in \mathfrak{h}$ and
\begin{equation}
    ad_{e_i}^{1-a_{ij}}(e_j)=0, \quad ad_{f_i}^{1-a_{ij}}(f_j)=0.
\end{equation}
Those are called the Chevalley-Serre generators and relations respectively. The constructed Lie algebra is called the Kac-Moody Lie algebra associated to the generalized Cartan matrix $A$ (\cite{Kac}).\\

%For this Kac-Moody Lie algebra the root lattice $Q$ is the abelian group
%\begin{equation}
%    Q:=\mathbb{Z}\alpha_1+\dots + \mathbb{Z}\alpha_n.
%\end{equation}
%The set $Q^+\subset Q$ is the subset of linear combinations of $\alpha_i$ with non-negative coefficients. One can also define the weight lattice $P$ by
%\begin{equation}
%    P:=\{\lambda\in \mathfrak{h}^*|\forall i \langle \lambda, \alpha_i^{\vee} \rangle \in \mathbb{Z}\}
%\end{equation}
%and the set of dominant weights $P^+\subset P$,
%\begin{equation}
%    P^+:=\{\lambda\in \mathfrak{h}^*|\forall i \langle \lambda, \alpha_i^{\vee} \rangle \ge 0\}.
%\end{equation}

\subsection{General linear groups and algebras.} The general linear group $GL_n(\CC)$ is the group of invertible matrices of size $n$ over $\CC$. The standard choice of a maximal torus $T_n$ of $GL_n(\CC)$ is the subgroup of diagonal matrices and the standard choice of a Borel subgroup $B_n$ is the subgroup of upper-triangular matrices. This yields the following description of the Lie algebra $\mathfrak{gl}_n(\CC)$ of $GL_n(\CC)$ and its root system:
\begin{equation}
    \mathfrak{gl}_n(\CC) = \mathfrak{n}_- \oplus \mathfrak{h} \oplus \mathfrak{n}_+, \quad \mathfrak{n}_-=\text{span}\langle E_{ij} \rangle_{1\le j<i \le n}, \quad  \mathfrak{n}_+=\text{span}\langle E_{ij} \rangle_{1\le i<j \le n}, 
\end{equation}
\begin{equation}
    \mathfrak{h}=\text{span}\langle E_{ii} \rangle_{1\le i \le n}, \quad \mathfrak{h}^*=\text{span}\langle \theta_i \rangle_{1\le i\le n},
\end{equation}
\begin{equation}
    \theta_i(E_{jj})=\delta_{ij}, \quad R=\{\theta_{i}-\theta_j|i\neq j\}, \quad R^+=\{\theta_i -\theta_j|i<j\},
\end{equation}
\begin{equation}
    \Pi=\{\theta_{i}-\theta_{i+1}\},\quad \Pi^{\vee}=\{E_{i,i}-E_{i+1,i+1}\},
\end{equation}
where $E_{ij}$ are the elementary matrices. We make the following choice for the standard Chevalley generators of $\mathfrak{gl}_n$:
\begin{equation}
    f_{i}=E_{i+1,i},\quad   e_{i}=E_{i,i+1}, \quad 1\le i\le n-1.
\end{equation}

All irreducible representations of $GL_n(\CC)$ (or, equivalently, integrable irreducible representations of $\mathfrak{gl}_n(\CC)$) are parameterized by $n$-tuples of integers $(\lambda_1,\dots,\lambda_n)$ such that $\lambda_i \ge \lambda_{i+1}$. If $V$ is the tautological representation of $GL_n(\CC)$ then any such representation can be tensored with the one-dimensional representation $\Lambda^n V$ several times, so that $\lambda$ becomes a partition of length not greater than $n$. The resulting representation may be realized via a Schur functor $\mathbb{S}^{\lambda}$ applied to $V$.
\subsection{The Deligne category.} The Deligne category $\Rep(GL,T)$ for a variable $T$ and the field $\CC$ is the Karoubi closure of the additive closure of the free rigid monoidal $\CC[T]$-linear category generated by an object $V$ of dimension $T$. For non-negative integers $n,m$ the endomorphism algebra of an object $V^{\otimes n} \otimes V^{*\otimes m}$ is the walled Brauer algebra $Br_{n,m}(T)$ over $\CC[T]$ (\cite{CW}).\\

For any element $t$ of $\CC$ we may specialize the category $\Rep(GL,T)$ to $T=t$. The resulting $\CC$-linear category $\Rep(GL_t)$ is also usually called a Deligne category (\cite{DM}). If $t$ is not an integer then $\Rep(GL_t)$ is abelian and semisimple (\cite{CW}[Theorem 4.8.1]). For integer $t$ it is only Karoubian (\cite{CW}).\\

Indecomposable objects $V_{[\lambda,\mu]}$ of $\Rep(GL_t)$ are parameterized by bi-partitions $(\lambda,\mu)$ and are obtained by applying appropriate idempotents to $V^{\otimes |\lambda|}\otimes V^{*\otimes |\mu|}$. For any positive integer $n\in \mathbb{N}$ the category $\Rep(GL_n)$ admits a full monoidal functor $F$ to the category Rep$(GL_n)$ of algebraic representations of $GL_n$ which sends $V$ to the tautological representation $V=\CC^n$ of $GL_n$ and $V_{[\lambda,\mu]}$ to the simple representation $V_{\lambda,\mu}$ in $\mathbb{S}^{\lambda} (\CC^n)\otimes \mathbb{S}^{\mu}(\CC^{n*})$ with the largest (w.r.t. the standard partial order on the root lattice) highest weight if $l(\lambda)+l(\mu) \le n$ where $\mathbb{S}^{\lambda}$ is the Schur functor corresponding to $\lambda$. If $l(\lambda)+l(\mu)>n$, then $F(V_{[\lambda,\mu]})=0$. From now on we will assume the notation $t=\dim V$ and distinguish between objects of $\Rep(GL_t)$ and objects of ${\rm Rep}(GL_t)$ depending on the context.\\

Let $GL_t$ be the fundamental group of $\Rep(GL_t)$ (\cite{D, ERT}). The Lie algebra $\mathfrak{gl}_t$ (or $\mathfrak{gl}(V)$) of $GL_t$ is
\begin{equation}
    \mathfrak{gl}_t= V\otimes V^*.
\end{equation}
Note that $\mathfrak{gl}_t$ is an associative algebra via the evaluation map, therefore it is also a Lie algebra (\cite{ERT}).\\

It is also possible to define the analogue $\Rep(O_t)$ of the Deligne categories in the case of the orthogonal group (\cite{Etc}) which is a $\mathbb{C}$-linear tensor category generated by an object $V$ of dimension $t$ with a symmetric isomorphism $V\xrightarrow{\sim} V^*$. This category is abelian and semisimple for $t\in \mathbb{C}\setminus \mathbb{Z}$ and Karoubian for $t\in \mathbb{Z}$. All irreducible objects $V_{[\lambda]}$ of $\Rep(O_t)$ are parametrized by partitions $\lambda$ and correspond to the respective summand of $\mathbb{S}^{\lambda}V $. For a positive integer $t$ $\Rep(O_t)$ admits a full monoidal functor $F$ to Rep$(O_t)$. If $t$ is negative even (negative odd), then after a change of the symmetry morphism by a sign (so that $V$ is treated as an odd object) $\Rep(O_t)$ admits a similar functor to the category of representations Rep$(Sp_{-t})$ of the symplectic group (Rep$(Osp(1|1-t))$ respectively).

\subsection{Skein categories.}
Similarly to the classical case, we have means to interpolate the category of representations of quantum groups. Namely, in the case of $U_q(\mathfrak{gl}_n)$ when $q$ is not a root of unity, the so-called oriented Skein category $\Rep_q(GL_t)$ is defined in \cite{Br}. It is the additive Karoubi envelope of a strict monoidal $\mathbb{C}$-linear category (although it can be defined over any field $\Bbbk$) generated by an object $X_q$ with its right dual $\leftindex^*X_q$ and maps between tensor products of $X_q$ and $\leftindex^*X_q$ are given by framed oriented tangles subject to relations
%\begin{equation}
%    S^2=(q-q^{-1})S+1_{X_q} \otimes 1_{X_q}, 
%\end{equation}
%\begin{equation}
%    (S \otimes 1_{X_q}) \circ (1_{X_q} \otimes S) \circ (S \otimes 1_{X_q}) = (1_{X_q} \otimes S) \circ (S \otimes 1_{X_q}) \circ (1_{X_q} \otimes S)
%\end{equation}
%\begin{equation}
%     (D \otimes 1_{X_q})\circ (1_{X_q} \otimes C) = 1_{X_q}, (1_{\leftindex^*X_q} \otimes D) \circ (C \otimes 1_{\leftindex^*X_q}) =1_{\leftindex^*X_q},
%\end{equation}
%\begin{equation}
%     T^{-1} = (1_{\leftindex^*X_q} \otimes 1_{X_q} \otimes D) \circ (1_{\leftindex^*X_q} \otimes S \otimes 1_{\leftindex^*X_q} ) \circ (C \otimes 1_{X_q} \otimes 1_{\leftindex^*X_q} ),
%\end{equation}
%\begin{equation}
%    q^tD \circ T \circ C =\frac{q^t-q^{-t}}{q-q^{-1}}1_{\bbone}.
%\end{equation}
\begin{equation}
\vcenter{\hbox{\begin{tikzpicture}
\coordinate (A) at (0,-1);
\coordinate (B) at (-0.6,0.4);
\coordinate (C) at (0,1);
\coordinate (D) at (-0.6,-0.4);
\draw [red,line width=1pt] (A) to[out=90,in=0] (B);
\draw [red,,line width=1pt] (B) + (0:0) arc (90:270:0.4);
\draw [white,line width=5pt] (D) to[out=0,in=-90] (C);
\draw [->,red,line width=1pt] (D) to[out=0,in=-90] (C);
\end{tikzpicture}}}=
\vcenter{\hbox{\begin{tikzpicture}
\coordinate (A) at (0,-1);
\coordinate (B) at (0.6,0.4);
\coordinate (C) at (0,1);
\coordinate (D) at (0.6,-0.4);
\draw [red, line width=1pt] (B) + (0:0) arc (90:-90:0.4);
\draw [->,red, line width=1pt] (D) to[out=180,in=-90] (C);
\draw [white,line width=5pt] (A) to[out=90,in=180] (B);
\draw [red, line width=1pt] (A) to[out=90,in=180] (B);
\end{tikzpicture}}}= q^t\cdot
\vcenter{\hbox{\begin{tikzpicture}
\draw [->,red, line width=1pt] (0,-1) to (0,1);
\end{tikzpicture}}}\, , \quad
\end{equation}
\begin{equation}
\vcenter{\hbox{\begin{tikzpicture}[
every braid/.style={
 line width=1pt,
braid/strand 1/.style={red},
braid/strand 2/.style={red},
}
]
\pic [red, line width=1pt] (coords) {braid={s_1 s_1^{-1}}};
\coordinate (A) at  (coords-1-s);
\coordinate (B) at  (coords-2-s);
\coordinate (C) at  (coords-1-e);
\coordinate (D) at  (coords-2-e);
\draw[->,red, line width=1pt] (A) to[out=90,in=-90] (A);
\draw[->,red, line width=1pt] (B) to[out=90,in=-90] (B);
\end{tikzpicture}
}}=
\vcenter{\hbox{\begin{tikzpicture}
\draw[->,red, line width=1pt] (0,-1.5) to[out=90,in=-90] (0,1);
\draw[->,red, line width=1pt] (0.5,-1.5) to[out=90,in=-90] (0.5,1);
\end{tikzpicture}
}}\, , \quad
\vcenter{\hbox{\begin{tikzpicture}[
every braid/.style={
 line width=1pt,
braid/strand 1/.style={red},
braid/strand 2/.style={red},
}
]
\pic [red, line width=1pt] (coords) {braid={s_1^{-1} s_1}};
\coordinate (A) at  (coords-1-s);
\coordinate (B) at  (coords-2-s);
\coordinate (C) at  (coords-1-e);
\coordinate (D) at  (coords-2-e);
\draw[->,red, line width=1pt] (A) to[out=90,in=-90] (A);
\draw[->,red, line width=1pt] (B) to[out=90,in=-90] (B);
\end{tikzpicture}
}}=
\vcenter{\hbox{\begin{tikzpicture}
\draw[->,red, line width=1pt] (0,-1.5) to[out=90,in=-90] (0,1);
\draw[->,red, line width=1pt] (0.5,-1.5) to[out=90,in=-90] (0.5,1);
\end{tikzpicture}
}}\, , \quad
\vcenter{\hbox{\begin{tikzpicture}[
every braid/.style={
 line width=1pt,
braid/strand 1/.style={->,red},
braid/strand 2/.style={red},
}
]
\pic [red, line width=1pt] (coords) {braid={s_1 s_1^{-1}}};
\coordinate (A) at  (coords-1-s);
\coordinate (B) at  (coords-2-s);
\coordinate (C) at  (coords-1-e);
\coordinate (D) at  (coords-2-e);
\draw[->,red, line width=1pt] (B) to[out=90,in=-90] (B);
\end{tikzpicture}
}}=
\vcenter{\hbox{\begin{tikzpicture}
\draw[->,red, line width=1pt] (0,1) to[out=-90,in=90] (0,-1.5);
\draw[->,red, line width=1pt] (0.5,-1.5) to[out=90,in=-90] (0.5,1);
\end{tikzpicture}
}}\, , \quad
\vcenter{\hbox{\begin{tikzpicture}[
every braid/.style={
 line width=1pt,
braid/strand 1/.style={red},
braid/strand 2/.style={->,red},
}
]
\pic [red, line width=1pt] (coords) {braid={s_1^{-1} s_1}};
\coordinate (A) at  (coords-1-s);
\coordinate (B) at  (coords-2-s);
\coordinate (C) at  (coords-1-e);
\coordinate (D) at  (coords-2-e);
\draw[->,red, line width=1pt] (A) to[out=90,in=-90] (A);
\end{tikzpicture}
}}=
\vcenter{\hbox{\begin{tikzpicture}
\draw[->,red, line width=1pt] (0,-1.5) to[out=90,in=-90] (0,1);
\draw[->,red, line width=1pt] (0.5,1) to[out=-90,in=90] (0.5,-1.5);
\end{tikzpicture}
}}\, ,
\end{equation}
\begin{equation}
\vcenter{\hbox{\begin{tikzpicture}[
every braid/.style={
 line width=1pt,
braid/strand 1/.style={red},
braid/strand 2/.style={red},
braid/strand 3/.style={red},
}
]
\pic [red, line width=1pt] (coords) {braid={s_1 s_2 s_1}};
\coordinate (A) at  (coords-1-s);
\coordinate (B) at  (coords-2-s);
\coordinate (C) at  (coords-3-s);
\draw[->,red, line width=1pt] (A) to[out=90,in=-90] (A);
\draw[->,red, line width=1pt] (B) to[out=90,in=-90] (B);
\draw[->,red, line width=1pt] (C) to[out=90,in=-90] (C);
\end{tikzpicture}
}}=
\vcenter{\hbox{\begin{tikzpicture}[
every braid/.style={
 line width=1pt,
braid/strand 1/.style={red},
braid/strand 2/.style={red},
braid/strand 3/.style={red},
}
]
\pic [red, line width=1pt] (coords) {braid={s_2 s_1 s_2}};
\coordinate (A) at  (coords-1-s);
\coordinate (B) at  (coords-2-s);
\coordinate (C) at  (coords-3-s);
\draw[->,red, line width=1pt] (A) to[out=90,in=-90] (A);
\draw[->,red, line width=1pt] (B) to[out=90,in=-90] (B);
\draw[->,red, line width=1pt] (C) to[out=90,in=-90] (C);
\end{tikzpicture}
}}\, ,
\end{equation}
\begin{equation}
\vcenter{\hbox{\begin{tikzpicture}
\draw[->,red, line width=1pt] (1,0) to (0,1);
\draw[white, line width=5pt] (0,0) to (1,1);
\draw[->,red, line width=1pt] (0,0) to (1,1);
\end{tikzpicture}
}}-
\vcenter{\hbox{\begin{tikzpicture}
\draw[->,red, line width=1pt] (0,0) to (1,1);
\draw[white, line width=5pt] (1,0) to (0,1);
\draw[->,red, line width=1pt] (1,0) to (0,1);
\end{tikzpicture}
}}=
z\cdot \vcenter{\hbox{\begin{tikzpicture}
\draw[->,red, line width=1pt] (0,0) to (0,1);
\draw[->,red, line width=1pt] (0.5,0) to (0.5,1);
\end{tikzpicture}
}}\, , \quad 
\vcenter{\hbox{\begin{tikzpicture}
\draw[->,red, line width=1pt] (0,0) arc (360:0:0.5);
\end{tikzpicture}
}}= \frac{q^t-q^{-t}}{q-q^{-1}}\cdot \text{Id},
\end{equation}
where we choose a branch of $q^t$. The composition rule is given by the concatenation of respective tangles.\\

All indecomposable objects of $\Rep_q(GL_t)$ are again parameterized by bipartitions. The category $\Rep_q(GL_t)$ is semisimple iff $q$ is not a root of unity and $q^t\neq \pm q^{n}$ for $n\in \mathbb{Z}$.\\

It is also possible to define the Birman-Wenzl-Murakami category $\Rep_q(O_t)$ which interpolates representation category of $U_q(\mathfrak{so}_n)$ by taking the additive Karoubi envelope of a strict monoidal $\mathbb{C}$-linear category generated by a self dual object $X_q$ with symmetric isomorphism $X_q \xrightarrow{\sim} X_q^{*}$ and morphisms given by tangles subject to BWM relations (e.g. see \cite{Ba} and \cite{LZ15}).

\section{KZ equations and dynamical differential equations}
% \begin{itemize}
%     \item demonstrate the identity between $\tilde{\Omega}_{ij}$ and $\kappa_{ij}$
%     \item state the duality theorem
% \end{itemize}
\subsection{Knizhnik-Zamolodchikov equations.} Consider the category $\Rep (GL_t)$ for a complex $t$. For integer $m,n\ge 0$ we may consider the Casimir operators
\begin{equation}\label{casimirdef}
    \Omega_{ij}: V^{*\otimes n}\otimes V^{\otimes m }\rightarrow V^{*\otimes n}\otimes V^{\otimes m }, \quad \Omega_{ij}=\Omega_{ji}
\end{equation}
which act in $i,j$ tensor components via a flip if $i,j\le n$ or $i,j> n$, and via $-\text{coev}\circ \text{ev}$ for other $i,j$. Here, ev$:V\otimes V^* \rightarrow \bbone$ and coev$:  \bbone\rightarrow V\otimes V^*$ are the evaluation and coevaluation maps.\\

We define the Knizhnik-Zamolodchikov connection on the base space
\begin{equation}
    \{(z_1,\dots,z_{n+m})\in \mathbb{C}^{n+m}|z_i\neq z_j \text{ for } i\neq j\}
\end{equation}
with values in $\Hom_{\Rep(GL_t)}(V_{[\lambda,\mu]}, V^{*\otimes n}\otimes V^{\otimes m })$
\begin{equation}
    \nabla_{KZ}(\hbar) = d -\hbar \sum_{i<j} \frac{dz_i -dz_j}{z_i-z_j}\Omega_{ij},
\end{equation}
where $\hbar\in \CC$ and the action of $\Omega_{ij}$ on $V^{*\otimes n}\otimes V^{\otimes m }$ is extended to endomorphisms of
\linebreak
$\Hom_{\Rep(GL_t)}(V_{[\lambda,\mu]}, V^{*\otimes n}\otimes V^{\otimes m })$. We may assume $|\lambda|+n=|\mu|+m$, otherwise 
\begin{equation}\label{Hom-spaces}
    \Hom_{\Rep(GL_t)}(V_{[\lambda,\mu]}, V^{*\otimes n}\otimes V^{\otimes m })=0.
\end{equation}

\begin{ex}[\cite{EV}]\label{Sm}
In the case when $\lambda,\mu =0$ and $m=n$ we can describe $\Omega_{ij}$ explicitly: note that $\Hom_{\Rep(GL_t)}(\bbone, V^{*\otimes m}\otimes V^{\otimes m})= \mathbb{C}[S_m]$, so for $1\le i< j\le 2m$ and $\sigma \in S_m$ we have
\begin{equation}\label{OmegaSm}
    \Omega_{ij}\sigma =
    \begin{cases}
        (i,j)\circ \sigma , & i,j\le m\\
        \sigma \circ (i-m,j-m), & i,j> m\\
        -t \sigma, & \sigma(j-m)=i, i\le m<j\\
        -(i,\sigma(j-m)) \circ \sigma, & \sigma(j-m)\neq i, i\le m<j
    \end{cases}%, \quad \sigma\in S_m
\end{equation}
\end{ex}

Since the vector space of homomorphisms $\Hom_{\Rep(GL_t)}(V_{[\lambda,\mu]}, V^{*\otimes n}\otimes V^{\otimes m })$ has the same dimension for all non-integer and all large enough integer $t=\dim V$, it is sufficient for us to consider the setup for $\mathfrak{gl}_t, t\in \mathbb{N}$ - for large $t$ we have an isomorphism
\begin{equation}\label{Deligne-Classic}
    F:\Hom_{\Rep(GL_t)}(V_{[\lambda,\mu]}, V^{*\otimes n}\otimes V^{\otimes m }) \xrightarrow{\sim} \Hom_{\mathfrak{gl}_t}(V_{\lambda,\mu}, V^{*\otimes n}\otimes V^{\otimes m }).
\end{equation}
Here $V_{\lambda,\mu}$ is the irreducible $\mathfrak{gl}_t$ representation of highest weight
\begin{equation}
    (\lambda_1, \lambda_2, \dots,0,\dots, 0, \dots, -\mu_2,-\mu_1 ),
\end{equation}
where the first coordinates are the coordinates of $\lambda$, the last coordinates are the coordinates of $-\mu$, and the coordinates in between are all zeros.\\

For large positive integer $t$ we also have
\begin{equation}\label{thespace}
    \Hom_{\mathfrak{gl}_t}(V_{[\lambda,\mu]}, V^{*\otimes n}\otimes V^{\otimes m })\cong \Hom_{\mathfrak{gl}_t}(V_{\lambda,\mu}\otimes (\Lambda^t V)^{\otimes n}, (\Lambda^{t-1}V)^{\otimes n}\otimes V^{\otimes m }).
\end{equation}

\subsection{$(\mathfrak{gl}_t ,\mathfrak{gl}_{n+m})$ duality.} Let $t$ be a positive integer and $\mathfrak{gl}_t$ the corresponding general linear Lie algebra. In this section we derive a duality between the KZ equations for $\mathfrak{gl}_t$ and dynamical differential equations for $\mathfrak{gl}_{n+m}$, via the joint action of $\mathfrak{gl}_t$ and $\mathfrak{gl}_{n+m}$ on the space $\Lambda^{\bullet} (V \otimes W),$ where $V,W$ are the tautological representations for $\mathfrak{gl}_t$ and $\mathfrak{gl}_{n+m}$ respectively. The derivation is similar to \cite{TL}.\\

The space in \eqref{thespace} can be given the structure of a weight space of a $\mathfrak{gl}(n+m)$-module. Namely, consider the space $\Lambda^{\bullet} (V \otimes W)$, which inherits the action of $\mathfrak{gl}(V)\oplus \mathfrak{gl}(W)$. Recall the following 

\begin{thm}[Skew-Howe duality,\cite{How}]
    As a $\mathfrak{gl}(V)\oplus \mathfrak{gl}(W)$-module
\begin{equation}\label{skew-Howe}
    \Lambda^{\bullet} (V \otimes W) = \bigoplus_{\delta, \ l(\delta) \le t, l(\delta^\top) \le n+m} V_{\delta} \otimes W_{\delta^\top},
\end{equation}
where $V_{\delta},W_{\delta^{\top}}$ are the irreducible representations of $\mathfrak{gl}(V)$ and $\mathfrak{gl}(W)$ of weights $\delta$ and $\delta^{\top}$ respectively. In this section we define $\delta^{\top}$ as the transpose of the partition $\delta$. The sum is over all partitions $\delta$ satisfying the conditions above.
\end{thm}

From the action of the Cartan subalgebra of $\mathfrak{gl}(W)$ it follows that given the standard weight basis for $W$, for any non-negative integer weight $\delta=(\delta_1,\dots,\delta_{n+m})$ we have an isomorphism of $\mathfrak{gl}(V)$-module and $\mathfrak{gl}(W)$-weight space of weight $\delta$:
\begin{equation}
    \Lambda^{\delta_1}V\otimes \dots \otimes \Lambda^{\delta_{n+m}}V \cong \Lambda^{\bullet}(V\otimes W)[\delta].
\end{equation}

In particular, we have an embedding $(\Lambda^{t-1}V)^{\tensor n} \tensor V^{\tensor m} \injto \Lambda^\bullet(V\tensor W)$ whose image is the subspace of $\mathfrak{gl}(n+m)$ weight 
\begin{equation}
    \beta := (\underbrace{t-1, \dots,t-1}_{n \text{ times}}, \underbrace{1, \dots,1}_{m \text{ times}}).
\end{equation}
Therefore, if $\mu_1\le n, \lambda_1\le m$ (otherwise the space \eqref{thespace} is $0$) we have an embedding
\begin{align}
    \Hom_{\mathfrak{gl}_t}(V_{\lambda,\mu}\otimes (\Lambda^t V)^{\otimes n}, & (\Lambda^{t-1}V)^{\otimes n}\otimes V^{\otimes m }) \notag \\ & \cong \Hom_{\mathfrak{gl}_t}(V_{\lambda,\mu}\otimes (\Lambda^t V)^{\otimes n}, \Lambda^{\bullet} (V \otimes W)[\beta])\cong W_{\gamma^\top}[\beta]
\end{align}
where $\gamma$ is the highest weight of the $\mathfrak{gl}_t$-module $V_{\lambda,\mu} \tensor (\Lambda^t V)^{\tensor n}$, % has the highest weight
\begin{equation}
    \gamma := (\underbrace{n+\lambda_1, n+\lambda_2, \dots, n-\mu_2, n-\mu_1}_{t\text{ entries}}).
\end{equation}

%We will identify the space $\Hom_{\Rep(GL_t)}(V_{\lambda,\mu}, V^{*\otimes n}\otimes V^{\otimes m })$ with the weight space $W_{\gamma^\top}[\beta]$ (Notice that the transpose $\gamma^\top$ can be interpolated to generic $t$; Lemma \ref{isomor} will extend the isomorphism to generic $t$). \\

Let us take the standard bases $v_a\in V$ and $w_{i}\in W$ for $1 \le a \le t, 1 \le i \le n+m$ and form a basis $x_{a,i}=v_a\otimes w_i$ of $V\otimes W$. As a $\mathfrak{gl}_t$-module, the space $\Lambda^{\bullet} (V\otimes W)$ is isomorphic to
\begin{equation}
    \Lambda^{\bullet}[x_{1,1},\dots,x_{t,1}] \otimes \cdots \otimes \Lambda^{\bullet} [x_{1,n+m},\dots,x_{t,n+m}].
\end{equation}
The $\mathfrak{gl}_t$ Casimir operators $\Omega_{ij}$ (as in \eqref{casimirdef}) act on this space as 
\begin{equation}\label{omegadef}
    \Omega_{ij} = \sum_{a} (e_a)_{(i)} (e^a)_{(j)}
\end{equation}
where $\{e_a\}, \{e^a\}$ are dual bases of $\mathfrak{gl}_t$, and the outside subscripts $(i)$ indicate action on the $i$-th factor of the tensor product. As $\Lambda^{\bullet}(V\otimes W)$ is a $\mathfrak{gl}_{n+m}$-module, it carries an action by the operators $\kappa_{ij}$ for $1 \le i,j \le n+m$, $i \neq j$ defined by
\begin{equation}
    \kappa_{ij} := e_\alpha e_{-\alpha} + e_{-\alpha} e_{\alpha}
\end{equation}
where $\alpha$ is the root $\theta_i - \theta_j$ of $\mathfrak{gl}_{n+m}$ and $e_{\pm \alpha}$ are the corresponding root vectors $E_{ij}, E_{ji}$ respectively from $\mathfrak{g}_{\pm\alpha}\subset \mathfrak{gl}_{n+m}$ normalized by Tr$(e_{\alpha}e_{-\alpha})=1$.\\

Let $E_{ij}, \ 1 \le i,j \le n+m,$ be the standard basis of $\mathfrak{gl}_{n+m}$.

\begin{lem}\label{actionequality}
    For any $1 \le i < j \le n+m$, the equality
    \begin{equation}
        2\Omega_{ij} = - \kappa_{ij} + E_{ii} + E_{jj}
    \end{equation}
    holds as operators on $\Lambda^{\bullet} (V\otimes W)$.
\end{lem}

\begin{proof}
    The action of $\Omega_{ij}$ on $\Lambda^{\bullet} (V\otimes W)$ can be written as
    \begin{equation}
        \sum_{1 \le a,b \le t} x_{a,i} \partial_{b,i} x_{b,j} \partial_{a,j}
    \end{equation}
    where $x_{r,c}$ and $\partial_{r,c}$ are the operators of multiplication and differentiation by $x_{r,c}$ (with appropriate powers of $-1$). Similarly, the action of $\kappa_{ij}$ is 
    \begin{equation}
        \sum_{1 \le a,b \le t} (x_{a,i} \partial_{a,j} x_{b,j} \partial_{b,i} + x_{b,j} \partial_{b,i} x_{a,i} \partial_{a,j}).
    \end{equation}
    
    In view of the anticommutation relation $x_{a,i} \partial_{b,j} + x_{b,j} \partial_{a,i} = \delta_{a,b} \delta_{i,j}$, we have
    \begin{align}
        \kappa_{ij} &= \sum_{1 \le a,b \le t} x_{a,i} (-x_{b,j} \partial_{a,j} + \delta_{a,b}) \partial_{b,i} + x_{b,j} (-x_{a,i} \partial_{b,i} + \delta_{a,b}) \partial_{a,j} \\
        &= -2\Omega_{ij} + \sum_{1 \le a \le t} (x_{a,i}\partial_{a,i} + x_{a,j}\partial_{a,j}) = -2\Omega_{ij} + E_{ii} + E_{jj}
    \end{align}
    as desired.
\end{proof}

%Now let $\Omega_{ij}$ be defined as in \eqref{omegadef} but with $\{X_a\}, \{X^a\}$ now being dual bases of $\mathfrak{sl}_n$. 
For arbitrary weights $\lambda,\mu$ Let $M_{\lambda, \mu}$ be the subspace of $\Lambda^{\bullet} (V\otimes W)$ with $\mathfrak{gl}_t$-weight $\lambda$ and $\mathfrak{gl}_{n+m}$-weight $\mu$. As a consequence of the above lemma, we have the following theorem.

\begin{thm}\label{duality}
    A function $f: \{(z_1,\cdots , z_{n+m}) \in \mathbb{C}^{n+m} \mid z_i \neq z_j\} \to M_{\lambda, \mu}$ is a flat section of the KZ connection
    \begin{equation}
        \nabla_{KZ} =  d - \hbar\sum_{1 \le i < j \le n+m} \frac{dz_i-dz_j}{z_i - z_j} \Omega_{ij}
    \end{equation}
    if and only if the function $g = f \cdot \prod_{1 \le i < j \le n+m} (z_i - z_j)^{-(\mu_i+\mu_j)\hbar /2}$ is a flat section of the connection
    \begin{equation}
        \nabla_{\kappa} := d + \frac{\hbar}{2}\sum_{1 \le i < j \le n+m} \frac{dz_i-dz_j}{z_i - z_j} \kappa_{ij}.
    \end{equation}
    Additionally, by using the gauge transformation $\nabla_{\kappa} \to h\nabla_{\kappa}h^{-1}$ where
    \begin{equation}
        h=\exp{\left(\frac{\hbar}{2}\sum_{1\le i<j\le n+m} (\mu_i-\mu_j)\log(z_i-z_j)\right)}
    \end{equation}
    we can change the $\nabla_{\kappa}$ connection to the dynamical connection $\nabla_{D}$ as in \cite{FMTV}:
    \begin{equation}\label{dynamicalconnection}
        \nabla_D = d+ \hbar \sum_{1 \le i < j \le n+m} \frac{dz_i-dz_j}{z_i - z_j} E_{ji}E_{ij}.
    \end{equation}
\end{thm}

\begin{proof}
    A straightforward computation. For the second part, note that
    \begin{equation}
        e_{\alpha}e_{-\alpha}+e_{-\alpha}e_{\alpha}=E_{ij}E_{ji}+E_{ji}E_{ij}=2E_{ji}E_{ij} +h_{\alpha}
    \end{equation}
    and $h_{\alpha}$ acts on $M_{\lambda,\mu}$ by $\mu_i-\mu_j$, where $\alpha = \theta_i -\theta_j$.
\end{proof}

In the the following lemma we assume that $t$ is either a non-integer or a large enough integer number. However, in the actual proof we will only need to know it for large enough integer numbers, because we can apply interpolation argument.

\begin{lem}\label{isomor}
    For all non-integer and large enough integer $t$ we have an isomorphism
    \begin{equation}\label{iso-lem}
        \phi :\Hom_{\Rep(GL_t)}(V_{[\lambda,\mu]}, V^{*\otimes n}\otimes V^{\otimes m }) \cong W_{\gamma^\top}[\beta],
    \end{equation}
    where $W_{\gamma^\top}$ is the unique (infinite-dimensional for generic $t$) irreducible $\mathfrak{gl}_{n+m}$-module of the highest weight $\gamma^\top$.
\end{lem}
\begin{proof}
    In this proof we will refer to the RHS or the LHS of the equation \eqref{iso-lem}. Note that for the specified $t$ the dimension of the LHS is the same as the dimension of the same space for some large enough integer $t$. The dimension of the RHS is constant due to the following parabolic induction consideration. The module $W_{\gamma^{\top}}$ can be obtained by a quotient of Ind$_{\mathfrak{p}}^{\mathfrak{gl}(n+m)} (W_{\gamma^{\top}_n}\otimes W_{\gamma^{\top}_m})$ where $\mathfrak{p}$ is the block upper-triangular Lie subalgebra that contains $\mathfrak{gl}(n)\oplus \mathfrak{gl}(m)$ and $W_{\gamma^{\top}_n}$ or $W_{\gamma^{\top}_m}$ are finite-dimensional irreducible $\mathfrak{gl}(n)$ or respectively $\mathfrak{gl}(m)$ modules with highest weights $\gamma_n^{\top}:=(\gamma^\top_1,\dots, \gamma^\top_n)$ and $\gamma_m^{\top}:=(\gamma^\top_{n+1},\dots, \gamma^\top_{n+m})$. The action of the nilpotent part of $\mathfrak{p}$ on $W_{\gamma^{\top}_n}\otimes W_{\gamma^{\top}_m}$ is trivial.\\
    
    For the choice of the LHS basis let us consider a projection
    \begin{equation}
        \pi: \Hom_{\Rep(GL_t)}(V^{\otimes|\lambda|}\otimes V^{*\otimes |\mu|}, V^{*\otimes n}\otimes V^{\otimes m }) \twoheadrightarrow \Hom_{\Rep(GL_t)}(V_{[\lambda,\mu]}, V^{*\otimes n}\otimes V^{\otimes m }).
    \end{equation}
    Recall the set of spanning $(w,w')-$diagrams from the bigger space as in \cite{CW} for a variable $T$. Since $\mathbb{C}[T]$ is a PID, it follows from the properties of the walled Brauer algebra that the LHS in \eqref{iso-lem} admits a set of vectors polynomial in $T$ such that it forms a basis after taking a quotient $T=t$ for $t\in \mathbb{C}\setminus \mathbb{Z}$ or for large enough integer $t$.\\
    
    Let us fix a basis of the corresponding weight space from the PBW-spanning set on the RHS. Note that the relations on the PBW vectors from the RHS are independent of $t$. Indeed, otherwise it would mean that we have a singular vector above $\beta$ in the Verma module $M_{\gamma^{\top}}$ whose coefficients necessarily depend on $t$. This in turn would imply the same fact for all large integer $t$, but with this assumption all the relations on the PBW vectors are independent of $t$ due to the consideration of the embedding below. In particular, it is clear that if a subset from the set of spanning PBW vectors is a basis for some $t$ as in the lemma, then it will also be a basis for the same weight space for all non-integer or large enough integer $t$ because the weights of the singular vectors of the corresponding Verma module sitting above $\beta$ are all the same for such $t$.\\
    
    For a large integer $t$ we may associate the space \eqref{thespace} with the space of $\mathfrak{gl}(V)$ highest weight vectors of the $\mathfrak{gl}(V)\oplus \mathfrak{gl}(W)$-weight $(\gamma,\beta)$ in $\Lambda^{\bullet}(V\otimes W)$. We may embed both spaces for a large integer $t$ into the $(\gamma,\beta)$-weight space of $\Lambda^{\bullet}(V\otimes W)$. In turn, it can be viewed as the space
    \begin{equation}\label{W-decomp}
        (\Lambda^{\gamma_1} W \otimes \cdots \otimes \Lambda^{\gamma_t} W) [\beta].
    \end{equation}
    The highest weight vector of $\Lambda^{\gamma_1} W \otimes \cdots \otimes \Lambda^{\gamma_t} W$ is already $\mathfrak{gl}(V)$-singular when it is embedded back into $\Lambda^{\bullet}(V\otimes W)$, so if we want to get the image of $\phi$ from \eqref{iso-lem} in \eqref{W-decomp}, it is sufficient for us to apply chains of $\mathfrak{gl}(W)$ lowering operators to this vector, so that we arrive in the correct weight space of weight $\beta$. The space in \eqref{W-decomp} has a spanning set
    \begin{equation}\label{W-basis}
        w_J = f_{j_1^1}\dots f_{j_{k_1}^1}w_1 \otimes \cdots \otimes f_{i_1^t}\dots f_{j_{k_t}^t}w_t,
    \end{equation}
    for various sequences $J=(j_1^1,\dots,j_{k_1}^1;\dots;i_1^t,\dots, j_{k_t}^t)$ and where $w_i$ are the highest weight vectors in $\Lambda^{\gamma_i}W$ and 
    \begin{equation}
        \text{wt}(f_{j_1^1}\dots f_{j_{k_1}^1}\dots f_{j_1^t}\dots f_{j_{k_t}^t}) = \gamma^\top-\beta.
    \end{equation}
    
    The coefficients in terms of \eqref{W-basis} of the basis from the LHS will be rational in $t$ and the coefficients of the basis from the RHS will be constant. We want to produce the matrix of the basis change from the RHS to the LHS. However, when $t\rightarrow +\infty$ the image of the LHS/RHS spaces lies in the subspace of a fixed (independent of $t$) finite dimension:
    \begin{equation}
        (\Lambda^{\gamma_1}W \otimes \cdots \Lambda^{\gamma_{|\lambda|}}W \otimes (\Lambda^n W \otimes \cdots \otimes \Lambda^n W)^{S_{t-|\lambda|-|\mu|}} \otimes \Lambda^{t-|\mu|+1} W \otimes \cdots \otimes \Lambda^{t} W) [\beta],
    \end{equation}
    where $S_{t-|\lambda|-|\mu|}$ acts by permutations of the tensor factors. Therefore, the matrix of the basis change has fixed rational coefficients in $t$ and we can identify the spaces from \eqref{iso-lem} for $t\in \CC\setminus \{-M,\dots,-1,0,1,\dots, M\}$ for some integer $M$.
\end{proof}

\begin{rem}
    Note that the weights $\beta, \gamma^{\top}$ are well defined for arbitrary $t$ as opposed to the weight $\gamma$. Indeed, in order to obtain $\gamma^\top$ we only have to modify the weight $(\underbrace{t,\dots,t}_{n \text{ times}})$ by adding or subtracting integer components of $\lambda^{\top},\mu^{\top}$ at the respective parts of $(t^n)$. This duality trick allows us to incorporate all $t$-dependence into $\beta $ and $ \gamma^{\top}$.
\end{rem}

We have the following consequence of \lemref{isomor}.

\begin{thm}\label{thm_isomor}
    The isomorphism $\phi$ in \eqref{iso-lem} identifies the dynamical connection on $W_{\gamma^\top}[\beta]$ and the KZ connection on
    $\Hom_{\Rep(GL_t)}(V_{[\lambda,\mu]}, V^{*\otimes n}\otimes V^{\otimes m })$. In particular, if 
    \begin{equation}
        f: \{(z_1,\cdots , z_{n+m}) \in \mathbb{C}^{n+m} \mid z_i \neq z_j\} \to W_{\gamma^\top}[\beta] 
    \end{equation}
    is a flat section of the dynamical connection, then
    \begin{equation}
        \prod_{1 \le i < j \le n+m} (z_i - z_j)^{-(\beta_i+\beta_j)\hbar /2} \cdot h\cdot  \phi^{-1}(f)=\prod_{1 \le i < j \le n+m} (z_i - z_j)^{-\beta_j\hbar } \cdot \phi^{-1}(f)
    \end{equation}
    is a flat section of $\nabla_{KZ}$.
\end{thm}

\begin{rem}
According to this theorem, the KZ connection on $\Hom_{\Rep(GL_t)}(V_{\lambda,\mu}, 
V^{*\otimes n}\otimes V^{\otimes m })$
in the Deligne category $GL_t$  corresponds to the dynamical connection on the 
$\mathfrak{gl}_{m+n}$-module $W_{\gamma^\top}$ whose highest weight is determined by the Deligne parameter $t$.
\end{rem}

\begin{proof}
    Since we know that $E_{ji}E_{ij}$ act as truncated Casimirs on $W_{\gamma^\top}[\beta]$ for all sufficiently large integers $t$ and this action is polynomial in $t$ (in terms of a PBW basis), it follows that $E_{ji}E_{ij}$ will still act as truncated Casimirs for all non-integer and large enough integer $t$.
\end{proof}

\section{Solutions to dynamical differential equations}
\subsection{Integral formulas.}\label{definesols}
In this section we make the same assumptions for $t$ as in the Lemma \ref{isomor}. We wish to write explicit integral formulas for the flat sections of the $\nabla_{KZ}(\hbar)$ connection on $\Hom$-spaces \ref{Hom-spaces}. Due to Theorem \ref{duality} it suffices to find flat sections of the dynamical connection \eqref{dynamicalconnection} for the Lie algebra $\mathfrak{gl}_{n+m}$ and the weight space $W_{\gamma^\top}[\beta]$. Explicitly, we are looking for solutions $u: \{(z_1,\dots , z_{n+m}) \in \mathbb{C}^{n+m} \mid z_i \neq z_j\} \to W_{\gamma^\top}[\beta]$ to the equations
\begin{equation}
    d u = -\hbar \sum_{1 \le i < j \le n+m} \frac{dz_i-dz_j}{z_i-z_j}e_{-\alpha} e_{\alpha} u
\end{equation}
where $\alpha$ is the root $\theta_i - \theta_j$ of $\mathfrak{gl}_{n+m}$ and $e_{\pm \alpha}$ are the corresponding normalized root vectors.\\

From \cite{FMTV}[Theorems 3.1, 3.2] and \cite{SV}, we have integral solutions to these equations, which we will now describe. Let $f_i = E_{i+1,i} \in \mathfrak{gl}_{n+m}$ for $1 \le i \le n+m-1$ be the standard lowering operators; associated with them are the simple roots $\alpha_i = \theta_i - \theta_{i+1}$. Write $\gamma^\top - \beta$ as a sum of simple roots $\lambda = \sum_{i=1}^{n+m-1} m_i \alpha_i$ for some $m_i \in \mathbb{Z}_{\ge 0}$ (note that $\gamma^\top - \beta$ stabilizes for generic $t$, so the $m_i$ do too). Let $\ol{m} = \sum_{i=1}^{n+m-1} m_i$, and let $c$ be the unique non-decreasing function $\{1,\dots,\ol{m}\} \to \{1,\dots,n+m-1\}$ such that $|c^{-1}(i)| = m_i$ for all $1 \le i \le n+m-1$.\\

Let $x\in \CC$ and $t_1,\dots,t_{\ol{m}}$ be the coordinates on the space 
\begin{equation}
    (\CC\setminus \{x\})^{\ol{m}} \setminus \bigcup_{1\le i<j \le \ol{m}} \{t_i=t_j\}.
\end{equation}
For permutations $\sigma \in S_{\ol{m}}$ define the differential $\ol{m}$-forms 
\begin{equation}
    \omega_{\sigma}(x) = d\log(t_{\sigma(1)}-t_{\sigma(2)}) \wedge \cdots \wedge d\log(t_{\sigma(\ol{m}-1)}-t_{\sigma(\ol{m})}) \wedge d\log(t_{\sigma(\ol{m})}-x), \quad d=d_t, \quad \omega_{\sigma}:=\omega_{\sigma}(0).
\end{equation}
Also define the operator $f_{c(\sigma)} := f_{c(\sigma(1))} \cdots f_{c(\sigma(\ol{m}))}$.
Let $v$ denote the highest weight vector in $W_{\gamma^\top}$. The $W_{\gamma^\top}[\beta]$-valued differential $\ol{m}$-form $\omega$ is defined as
\begin{equation}\label{omega_form}
    \omega(t_1,\dots,t_{\ol{m}},x) = \sum_{\sigma \in S_{\ol{m}}} (-1)^{|\sigma|} \omega_{\sigma}(x) f_{c(\sigma)}v.
\end{equation}
We can also define $\omega=\omega(t_1,\dots,t_{\ol{m}},0)$.\\

Let us introduce the master function
\begin{equation}
    \widetilde{\Phi}(x):= \prod_{1 \le i \le \ol{m}} (t_i-x)^{-(\alpha_{c(i)},\gamma^\top)} \prod_{i < j} (t_i - t_j)^{(\alpha_{c(i)}, \alpha_{c(j)})}, \quad \Phi:=\widetilde{\Phi}(0).
\end{equation}

\begin{thm}[\cite{FMTV, SV}]\label{FMTV_sol}
    For any appropriate cycle 
    \begin{equation}
        \Gamma\in H_{\overline{m}}((\CC\setminus \{x\})^{\overline{m}} \setminus \bigcup_{1\le i<j \le \overline{m}} \{t_i=t_j\})
    \end{equation}
    the sections
    \begin{equation}
        u(z_1,\dots,z_{n+m},x):=\int_{\Gamma} \exp{\left(\hbar\left(\sum_{i=1}^{\ol{m}} \left(z_{c(i)}-z_{c(i)+1}\right) t_i - \langle \gamma^{\top}, z\rangle x\right)\right)}\widetilde{\Phi}^{-\hbar}(x) \omega(t,x)
    \end{equation}
    satisfy both the trivial Knizhnik-Zamolodchikov connection (in the single variable $x$)
\begin{equation}
    d+\hbar z, \quad d=d_x
\end{equation}
where we view $z=(z_1,\dots,z_{n+m})$ as an element of the standard Cartan subalgebra of $\mathfrak{gl}_{n+m}$, and the dynamical equations for 
\begin{equation}\label{full_dynamical_connection}
    \nabla_D':= d_z +\hbar \left(\sum_{i=1}^{n+m}\beta_i x dz_i +\sum_{1\le i<j\le n+m}\frac{dz_i-dz_j}{z_i-z_j}e_{-\alpha}e_{\alpha} \right).
\end{equation}
\end{thm}

We have the following consequence of this theorem.

\begin{cor}
    The sections $u(z_1,\dots,z_{n+m},0)$ satisfy the dynamical equations for $\nabla_D$ from \thmref{duality}.
\end{cor}

Let us fix an ordering $l:\{1,\dots,\ol{m}\} \to \{1,\dots,\ol{m}\}$ of the set $\{1,\dots,\ol{m}\}$. For each $l$ we construct a cycle in $\Gamma_l\in H_{\overline{m}}((\CC\setminus \{0\})^{\overline{m}} \setminus \bigcup_{1\le i<j \le \overline{m}} \{t_i=t_j\})$ given by the picture below,

\begin{center}
    \begin{tikzpicture}

  \filldraw (0,0) circle (2pt);
  \node at (0.2,0.2) {$0$};
  
  \draw[xshift=5,black!60!black,decoration={markings,mark=between positions 0.125 and 0.875 step 0.25 with \arrow{>}},postaction={decorate}] (5,-0.5) -- (-0.5,-0.5) arc (-90:-270:0.5) (-0.5,0.5) -- (5,0.5) node[above left] {$t_{l^{-1}(1)}$};

  \node at (3.5,1.5) {$\ddots$};

  \draw[xshift=5,black!60!black,decoration={markings,mark=between positions 0.125 and 0.875 step 0.25 with \arrow{>}},postaction={decorate}] (5,-1) -- (-1,-1) arc (-90:-270:1.5) (-1,2) -- (5,2) node[above left] {$t_{l^{-1}(\overline{m})}$};
\end{tikzpicture}

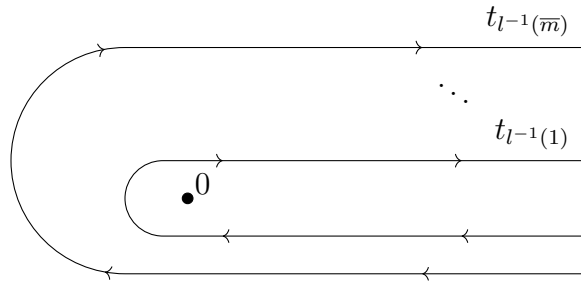
\upshape \textbf{Pic. 1.} Integration contours for $\Gamma_l$.
\end{center}

\begin{thm}\label{solthm}
     The sections of the form
\begin{equation}\label{intsols}
    u_l := \int_{\Gamma_l} \exp\left(\hbar\sum_{i=1}^{\ol{m}} \left(z_{c(i)}-z_{c(i)+1}\right) t_i\right)\Phi^{-\hbar} \omega
\end{equation}
span the space of solutions of the dynamical equations in $W_{\gamma^\top}[\beta]$. The integrals converge in the region ${\rm Re} (\hbar(z_i -z_{i+1}))<0$.
\end{thm}

\begin{proof}
    Consider a limit $\hbar=\epsilon, z_i = z_i'/\epsilon, \epsilon \to 0$ so that ${\rm Re}(z_i'-z_{i+1}')<0$. By deforming the contours of integration we may assume that both ``tails'' of each individual contour are close to the real line.
    
%\pagebreak

\begin{center}
    \begin{tikzpicture}

  \filldraw (0,0) circle (2pt);
  \node at (0.2,0.2) {$0$};
  \draw[thick,->] (-0.5,0) -- (7,0); %node[above] {Re};
  
  \draw[black!60!black,decoration={markings,mark=between positions 0.125 and 0.875 step 0.25 with \arrow{>}},postaction={decorate}] (5,-0.2) -- (0.4,-0.2) arc (-26.6:-333.4:0.44) (0.4,0.2) -- (5,0.2);
  \node at (5.8,0.3) {$t_{l^{-1}(1)}$};

  \node at (5.8,0.6) {$\dots$};

  \draw[black!60!black,decoration={markings,mark=between positions 0.125 and 0.875 step 0.25 with \arrow{>}},postaction={decorate}] (5,-0.4) -- (0.8,-0.4) arc (-26.6:-319:1) (0.65,0.71) -- (5,0.71);
  \node at (5.8,0.9) {$t_{l^{-1}(\overline{m})}$};
\end{tikzpicture}

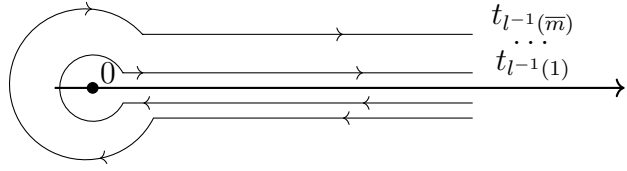
 \textbf{Pic. 2.} Deformation of the cycle $\Gamma_l$.
\end{center}
 We may note that the integral $u_l(z_i',\epsilon)$ converges absolutely in $\epsilon$, so it is holomorphic in $\epsilon$ and we may consider $u_l(z_i',0)$. Assume for simplicity that we are working with only one term $\omega_{\sigma}$ of $\omega$. When we let $\epsilon=0$, the function under the integral \eqref{intsols} becomes holomorphic on $\CC^{\overline{m}} \setminus (\bigcup_{i< j}\{t_i=t_j\} \cup \bigcup_{i} \{t_i=0\})$. If we look at the function $u_l(z_i',0)$ without the integral over $t_{l^{-1}(1)}$, the resulting function $f_1(t_{l^{-1}(1)},z_i')$ is either holomorphic or meromorphic in $t_{l^{-1}(1)}$ with the only simple pole at $t_{l^{-1}(1)}=0$ depending on the last factor of $\omega_{\sigma}$ in the denominator. Therefore, if we perform the missing integration in $t_{l^{-1}(1)}$ and pinch two tails of integration from $a+i0+$ to $+\infty+i0+$ and back from $+\infty-i0+$ to $a-i0+$ where $a\in \mathbb{R}_{>0}$, they will cancel each other out. Thus the resulting integral computes the residue of $f_1(t_{l^{-1}(1)},z_i')$ at $t_{l^{-1}(1)}=0$. If $\omega_{\sigma}$ does not have a pole at $t_{l^{-1}(1)}=0$, then the integral is zero.\\

 This argument shows that we may algebraically compute the residue of the function
 \begin{equation}
     \exp{\left(\sum_{i=1}^{\overline{m}} (z_{c(i)}'-z_{c(i)+1}')t_i\right)} \frac{1}{(t_{\sigma(1)}-t_{\sigma(2)})(t_{\sigma(2)}-t_{\sigma(3)})\dots (t_{\sigma(\overline{m}-1)}-t_{\sigma(\overline{m})})t_{\sigma(\overline{m})}}
 \end{equation}
at $t_{l^{-1}(1)}=0$ and then perform the other $\overline{m}-1$ integrations. If $\sigma(\overline{m})=l^{-1}(1)$, the residue is equal to
\begin{equation}
    \exp{\left(\sum_{i=1}^{\overline{m}-1} (z_{c(i)}'-z_{c(i)+1}')t_i\right)} \frac{1}{(t_{\sigma(1)}-t_{\sigma(2)})(t_{\sigma(2)}-t_{\sigma(3)})\dots (t_{\sigma(\overline{m}-2)}-t_{\sigma(\overline{m}-1)})t_{\sigma(\overline{m}-1)}}.
\end{equation}

Now we can consider the same argument for the next variable $t_{l^{-1}(2)}$ and so on. From this we see that 
\begin{equation}
    \int_{\Gamma_l} \exp{\left(\sum_{i=1}^{\overline{m}} (z_{c(i)}'-z_{c(i)+1}')t_i\right)} \omega_{\sigma} = (-2\pi {\rm i})^{\overline{m}}\delta_{\sigma, l^{-1}\circ w},
\end{equation}
where $w(i)=\overline{m}+1-i, 1\le i \le \overline{m}$. Then we have
\begin{equation}
    u_l(z_i',0) = \int_{\Gamma_l} \sum_{\sigma \in S_{\ol{m}}} \omega_{\sigma} f_{c(\sigma)}v = (-2\pi {\rm i})^{\overline{m}}f_{c(l^{-1}\circ w)}v.
\end{equation}
The vectors $f_{c(\sigma)}v$ span $W_{\gamma^\top}[\beta]$, so the solutions \eqref{intsols} span the space of all solutions in $W_{\gamma^\top}[\beta]$.
\end{proof}

\begin{rem}
    For large integer $t$ there is a natural embedding
    \begin{equation}
        \varphi : W_{\gamma^\top}\to \Lambda^{\gamma_1}W \otimes \cdots \otimes \Lambda^{\gamma_t}W
    \end{equation}
    which sends the highest-weight vector of $W_{\gamma^\top}$ to the product of highest-weight vectors. One might also try to write solutions for large integer $t$ using Theorem 3.1 in \cite{FMTV} on the $\mathfrak{gl}_{n+m}$ weight space $\left(\Lambda^{\gamma_1}W \otimes \cdots \otimes \Lambda^{\gamma_t}W\right)[\beta]$. However, we can show that the solutions obtained in this way actually lie in the image of $\varphi$, and are in fact the same as the solutions obtained in Theorem \ref{solthm}. Explicitly, the ``new'' solutions are described as follows: let $P$ be the set of sequences $\sigma=(i_1^1,\dots,i_{s_1}^1;\dots ;i_1^t,\dots,i_{s_t}^{t})$ consisting of the numbers $1,\dots,\ol{m}$ arranged into $t$ rows. For each such sequence, define the differential form $\omega_{\sigma} = \omega_{i_1^1,\dots,i_{s_1}^1} \wedge \cdots \wedge \omega_{i_1^t,\dots,i_{s_t}^t}$ where $\omega_{i_1,\dots,i_s} := d\log(t_{i_1}-t_{i_2}) \wedge \dots \wedge d\log (t_{i_{s-1}}-t_{i_s})\wedge d\log (t_{i_s})$. Also define the vector $f_{\sigma}v := f_{c(i_1^1)} \cdots f_{c(i_{s_1}^1)} v_1 \tensor \cdots \tensor f_{c(i_1^t)} \cdots f_{c(i_{s_t}^t)} v_t$ where $v_j$ is the highest-weight vector in $\Lambda^{\gamma_j}W$. Then the ``new'' solutions are given by
    \begin{equation}\label{newintsols}
        u_l = \int_{\Gamma_l} \exp\left(\hbar\sum_{i=1}^{\ol{m}} \left(z_{c(i)}-z_{c(i)+1}\right) t_i\right)\Phi^{-\hbar} \widetilde{\omega}
    \end{equation}
    where
    \begin{equation}
        \widetilde{\omega} := \sum_{\sigma \in P} (-1)^{|\sigma|} \omega_{\sigma} f_{\sigma}v.
    \end{equation}
    Here $|\sigma|$ is the sign of the permutation associated to $\sigma$ -- note that the sequence \linebreak $(i_1^1,\dots,i_{s_1}^1,\dots ,i_1^t,\dots,i_{s_t}^{t})$ is related to a standard sequence $(1,\dots, \ol{m})$ by a unique permutation. By repeatedly using Lemma 7.4.4 from \cite{SV} and the formula for the action of a Lie algebra $\mathfrak{g}$ on a tensor product of $\mathfrak{g}$-modules, we can re-arrange terms in \eqref{newintsols}. This way we can see that the solutions \eqref{newintsols} are the same as \eqref{intsols}.
\end{rem}

\begin{ex}\label{explicitSm}
    As in Example \ref{Sm}, consider the case when $\lambda,\mu=0$ and $m=n$ so we have $\Hom_{\Rep(GL_t)}(\mathbbm{1}, V^{*\otimes m}\otimes V^{\otimes m}) \cong \CC[S_m]$. This is also identified with the $\mathfrak{gl}_{2m}$-weight space $W_{\gamma^\top}[\beta]$ where
    \begin{equation}
        \gamma^\top = (\underbrace{t,\dots,t}_{m \text{ times}},\underbrace{0,\dots,0}_{m \text{ times}}).
    \end{equation}
    The difference $\gamma^\top-\beta$ is written as the sum of simple roots $\sum_{i=1}^{2m-1} m_i \alpha_i$ where $m_i = m-|m-i|$, so our solutions involve $\ol{m} = m^2$ integrations.
\end{ex}

\subsection{Bethe ansatz}
The  Bethe ansatz is a method to find joint eigenvectors of the Gaudin Hamiltonians $H_i = \sum_{j \neq i} \frac{\Omega_{ij}}{z_i-z_j}$
 which appear on the right-hand side of the KZ equations. One obtains such eigenvectors from 
 the integral representations of solutions to KZ equations  by taking the limit $\hbar \to 0$
and using the steepest descend method, see \cite{RV}.\\
 
More precisely, for any non-degenerate critical point of the function
\begin{equation}
    \exp{\left( \hbar \sum_{i=1}^{\overline{m}}(z_{c(i)}-z_{c(i)+1})t_i \right)}\Phi^{-\hbar}
\end{equation}
the value of the differential form $\omega$ from \eqref{intsols} at this point is a joint eigenvector, see \cite{RV}.
Such eigenvectors are called the Bethe vectors.\\

The critical set of this function has  not
been analyzed yet. It  not clear if the critical points of this function are non-degenerate and if their number is big enough so that the corresponding Bethe vectors diagonalize the Gaudin  Hamiltonians. \\

Nevertheless, we may still prove that the joint spectrum of the Gaudin Hamiltonians  is  simple. \\

 %Explicitly, if the spectrum of $H_i$ is simple, the Bethe vectors for $H_i$ acting on the RHS of \eqref{isomor} can be extracted from KZ solutions \eqref{intsols} by using the WKB method (\cite{RV}). For that we would need to find all the critical points of the function
%\begin{equation}
%    \exp{\left( \hbar \sum_{i=1}^{\overline{m}}(z_{c(i)}-z_{c(i)+1})t_i \right)}\Phi^{-\hbar}
%\end{equation}
%(given by taking the sections in Theorem \ref{solthm} and applying the transformations in Theorem \ref{duality}) can be written as a formal expansion in $\kappa$,
%\begin{equation}\label{KZexpansion}
%    f = e^{S/\kappa} \sum_{n=0}^{\infty} \kappa^{n} f_n.
%\end{equation}
%Substituting this into the KZ equations, we see that the leading term $f_0$ is an eigenvector, 
%\begin{equation}\label{eigvecs}
%    \sum_{j \neq i} \frac{\Omega_{ij}}{z_i-z_j} f_0 = (\partial_i S) f_0.
%\end{equation}
%In principle, we could compute the expansion \eqref{KZexpansion} by using the steepest descent method to evaluate the $\kappa \to 0$ limit of the integrals in \eqref{intsols}.\\ 
%However, it turns out that this function has many critical points which are difficult to analyze.\\

%Nevertheless, we may still prove that the spectrum is indeed simple. 
\begin{prop}
    The common spectrum of the Gaudin hamiltonians $H_i$ on 
    \begin{equation}\label{Hom-space}
        \Hom_{\Rep(GL_t)}(V_{[\lambda,\mu]}, V^{*\otimes n}\otimes V^{\otimes m })
    \end{equation}
    is simple for generic $t, z_i$.
\end{prop}
\begin{proof}
    %Via the relation to the KZ solutions, the Gaudin spectrum is formed by the eigenvalues $\partial_i S(z_1,\dots)$ where $S$ is defined as in \eqref{KZexpansion}. It is well known this spectrum is simple when $t$ is a positive integer (e.g.\ \cite{RV}); since the functions $S$ depend analytically on $t$, the spectrum must remain simple for generic $t$.
    The simplicity of the spectrum is a Zariski open condition on parameters $t,z_i$, so it is sufficient for us to prove it for a special $t$ and generic $z_i$. The latter can be proved by taking a sufficiently large integer $t$. In this case we have isomorphisms \eqref{Deligne-Classic}, \eqref{thespace} and \eqref{iso-lem}, so the space \eqref{Hom-space} can be identified with the space Sing$\left( \Lambda^{\bullet}(V\otimes W)_{\gamma, \beta} \right)$ of all $\mathfrak{gl}(V)$-singular vectors of $\mathfrak{gl}(V)\oplus \mathfrak{gl}(W)$ weight space $(\gamma, \beta)$ in $\Lambda^{\bullet}(V\otimes W)$. However, it follows from \cite{MV} that for generic $z_i$ Gaudin hamiltonians $H_i$ separate the Bethe vectors basis in Sing$\left( \Lambda^{\bullet}(V\otimes W)_{\gamma, \beta} \right)$, thus we have the proposition.
\end{proof}
\begin{rem}
    It follows from this proposition that the Gaudin hamiltonians $H_i$ generically (i.e. for generic $z_i,t$) generate the image of the Bethe algebra from \cite{FRU} in the space
    \begin{equation}
        \End(\Hom_{\Rep(GL_t)}(V_{[\lambda,\mu]}, V^{*\otimes n}\otimes V^{\otimes m })).
    \end{equation}
\end{rem}
\section{The orthogonal case}
Following Section 3, let $t$ be a positive integer and $\mathfrak{so}_t$ be the Lie algebra of the orthogonal group $O_t$ preserving the form $(e_i,e_j)=\delta_{ij}$ for an orthonormal basis $\{e_i\}$ of the tautological representation $V$. Consider the tensor product $V^{\otimes n}$. The KZ connenction for $\mathfrak{so}_{t}$ looks as follows:
\begin{equation}\label{KZ_O_t}
    \nabla_{KZ}=d-\hbar\sum_{1\le a<b \le n} \frac{dz_{a}-dz_{b}}{z_a-z_b}\Omega_{a,b}, 
\end{equation}
\begin{equation}
    \Omega_{a,b}:= \frac{1}{2}\sum_{1\le i<j\le t}(E_{ij}-E_{ji})^{(a)}\otimes (E_{ji}-E_{ij})^{(b)}.
\end{equation} 

Note that $\Omega_{a,b}$ commutes with $O_t$ - it is obvious for $SO_t$, so it is sufficient to check this for diag$(1,\dots,1,-1)$ (the adjoint action of this diagonal matrix is trivial on $\Omega_{a,b}$).\\

In the Deligne category $\Rep(O_t)$ the Casimir operator $\Omega_{1,2}:V\otimes V \to V\otimes V$ corresponds to the homomorphism $P-C$ where $1,P,C$ are given by the following diagrams\\

\begin{center}
    \begin{tikzpicture}

  \draw (0,0) circle (2pt);
  \draw (1,0) circle (2pt);
  \draw (0,-1) circle (2pt);
  \draw (1,-1) circle (2pt);
  \node at (0.5,-1.5) {$1$};
  \draw (0,0) -- (0,-1);
  \draw (1,0) -- (1,-1);

  \draw (3,0) circle (2pt);
  \draw (4,0) circle (2pt);
  \draw (3,-1) circle (2pt);
  \draw (4,-1) circle (2pt);
  \node at (3.5,-1.5) {$P$};
  \draw (3,0) -- (4,-1);
  \draw (4,0) -- (3,-1);
  
  \draw (6,0) circle (2pt);
  \draw (7,0) circle (2pt);
  \draw (6,-1) circle (2pt);
  \draw (7,-1) circle (2pt);
  \node at (6.5,-1.5) {$C$};
  \draw (6,0) to[out=-45,in=-145] (7,0);
  \draw (6,-1) to[out=45,in=145] (7,-1);

\end{tikzpicture}
\end{center}
This formula gives rise to the KZ connection for the space
\begin{equation}\label{Rep_O_t}
    \Hom_{\Rep(O_t)}(V_{[\lambda]},V^{\otimes n})
\end{equation}
for a partition $\lambda$.
\subsection{$(\mathfrak{so}_t,\mathfrak{so}_{2n})$ duality.} Consider the space $\Lambda^{\bullet}(V\otimes W), \dim(V)=t, \dim(W)=n$. We fix the following notations from \cite{BK}. Let $\mathfrak{so}_{2n}$ be the orthogonal Lie algebra preserving the form $(v_i,v_j)=\delta_{i,-j}$ on a space $\mathbb{C}^{2n}$ with the basis $v_i, i\in \{-n,\dots, -1, 1,\dots, n\}$. Let 
\begin{equation}
    M_{i,j}:=E_{i,j}-E_{-j,-i}.
\end{equation}

The space $\Lambda^{\bullet}(V\otimes W)$ admits an action of $\mathfrak{so}_{2n}$ via
\begin{equation}
    M_{i,j}^{(k)}=\frac{x_{ki}\partial_{kj}-\partial_{kj}x_{ki}}{2},\quad M_{-i,j}^{(k)}=\frac{\partial_{ki}\partial_{kj}-\partial_{kj}\partial_{ki}}{2},\quad M_{i,-j}^{(k)}=\frac{x_{ki}x_{kj}-x_{kj}x_{ki}}{2},\quad i,j>0.
\end{equation}

We have the following equality in terms of the Clifford algebra
\begin{align}\label{duality_o_t}
    &\Omega_{a,b}=\frac{1}{2}\sum_{1\le i<j\le t}(x_{ia}\partial_{ja}-x_{ja}\partial_{ia})(x_{jb}\partial_{ib}-x_{ia}\partial_{jb})=\\=\frac{1}{2}\sum_{i<j}&(-x_{ia}\partial_{ib} x_{jb}\partial_{ja}+x_{ja}x_{jb}\partial_{ia}\partial_{ib}+x_{ia}x_{ib}\partial_{ja}\partial_{jb}-x_{ja}\partial_{jb}x_{ib}\partial_{ia})=\\=\frac{1}{2}\sum_{i<j} (-& M_{a,b}^{(i)} M_{b,a}^{(j)}-M_{b,a}^{(i)}  M_{a,b}^{(j)}+M_{-a,b}^{(i)} M_{a,-b}^{(j)}+M_{a,-b}^{(i)} M_{-a,b}^{(j)})=\\
    =-\frac{1}{2}\sum_{i\neq j}(& M_{a,b}^{(i)} M_{b,a}^{(j)} -M_{-a,b}^{(i)}  M_{a,-b}^{(j)})=-\frac{1}{2}(M_{a,b}M_{b,a}+M_{-a,b}M_{b,-a}) +\\&+\frac{1}{2}\sum_{1\le i\le t}(M_{a,b}M_{b,a}+M_{-a,b}M_{b,-a})^{(i)}
\end{align}
where the last term is a sum of quadratic elements acting on $i$th $\mathfrak{gl}(W)$ tensor component. Note that the space $\Lambda^{\bullet}(V\otimes W)$ still retains decomposition 
\begin{equation}
    \Lambda^{\bullet}(V\otimes W)=\bigoplus_{\lambda} \Lambda^{\lambda_1}W \otimes\dots \otimes  \Lambda^{\lambda_t}W.
\end{equation}

We may identify the space \eqref{Rep_O_t} with the space of $\mathfrak{so}(V)$-singular (under some choice of Borel subalgebra) vectors of weight $\lambda$ and of $\mathfrak{so}_{2n}$ weight $\beta:=(1-\frac{t}{2},\dots, 1-\frac{t}{2})$ in the highest weight $\mathfrak{so}_{2n}$-module of weight $\gamma^{\top}:=(\lambda_1'-\frac{t}{2},\dots, \lambda_n'-\frac{t}{2})$ where $\lambda'$ is the usual transposition of a partition (note that now the notation $\gamma^{\top}$ differs from the one given in Section 3). Let us deal with the last term in \eqref{duality_o_t}.
\begin{align}
    &(M_{a,b}M_{b,a}+M_{-a,b}M_{b,-a})^{(i)} = x_{ia}\partial_{ib}x_{ib}\partial_{ia}+\partial_{ia}\partial_{ib}x_{ib}x_{ia}=\\
    =x_{ia}& \partial_{ia}(1-x_{ib}\partial_{ib})+(1-x_{ia}\partial_{ia})(1-x_{ib}\partial_{ib})=1-x_{ib}\partial_{ib}=\frac{1}{2}-M_{bb}^{(i)},
\end{align}
so
\begin{equation}
    \sum_{i=1}^t(M_{a,b}M_{b,a}+M_{-a,b}M_{b,-a})^{(i)}=\frac{t}{2}-M_{bb}.
\end{equation}

This allows us to rewrite the KZ connection \eqref{KZ_O_t} as follows:
\begin{align}
    \nabla_{KZ}=d+&\frac{\hbar}{2}\sum_{1\le a<b\le n}\frac{dz_a-dz_b}{z_a-z_b}(M_{b,a}M_{a,b}+M_{a,a}-M_{b,b}+M_{-a,b}M_{b,-a}-\frac{t}{2}+M_{b,b})=\\
    &=d+\hbar\sum_{1\le a<b\le n}\frac{dz_a-dz_b}{z_a-z_b}(e_{\epsilon_{b}-\epsilon_{a}}e_{\epsilon_{a}-\epsilon_{b}}+e_{-\epsilon_{a}-\epsilon_{b}}e_{\epsilon_{a}+\epsilon_{b}}-\frac{t}{4}+\frac{M_{a,a}}{2}).
\end{align}
Here, $e_{\alpha},\alpha\in R$ are the normalized root elements and $\epsilon_{a}$ are the standard diagonal elements of the dual space to the Cartan subalgebra of $\mathfrak{so}_{2n}$. After the restriction to the weight space in question we will get
\begin{equation}
    d+\hbar\sum_{1\le a<b\le n}\frac{dz_a-dz_b}{z_a-z_b}(e_{\epsilon_{b}-\epsilon_{a}}e_{\epsilon_{a}-\epsilon_{b}}+e_{-\epsilon_{a}-\epsilon_{b}}e_{\epsilon_{a}+\epsilon_{b}}+\frac{1-t}{2}).
\end{equation}

Let us fix a $\mathfrak{so}_t\oplus \mathfrak{so}_{2n}$ weight space $M_{\lambda, \mu}$ in $\Lambda^{\bullet}(V\otimes W)$ with weight $(\lambda,\mu)$. We obtain the following.
\begin{thm}\label{O_t_duality}
    A function $f: \{(z_1,\cdots , z_{n}) \in \mathbb{C}^{n} \mid z_i \neq z_j\} \to M_{\lambda, \mu}$ is a flat section of the KZ connection \eqref{KZ_O_t} if and only if the function 
    \begin{equation}
        g = f \cdot [\prod_{a<b}(z_a-z_b)]^{\hbar\frac{1-t}{2}}
    \end{equation}
    is a flat section of the dual flat connection
    \begin{equation}\label{new_connection}
        d+\hbar\sum_{1\le a<b\le n}\frac{dz_a-dz_b}{z_a-z_b}(e_{\epsilon_{b}-\epsilon_{a}}e_{\epsilon_{a}-\epsilon_{b}}+e_{-\epsilon_{a}-\epsilon_{b}}e_{\epsilon_{a}+\epsilon_{b}})
    \end{equation}
\end{thm}

In particular, the connection \eqref{new_connection} is flat.
%\begin{proof}
%    See Appendix B for the flatness check.
%\end{proof}
\begin{ex}
    Consider an isomorphism of Lie algebras $\mathfrak{so}(6)\cong\mathfrak{sl}(4)$ sending $M_{12}$, $M_{23}$,$M_{2,-3}$ to $E_{23},E_{34}$ and $E_{12}$ respectively. Under this isomorphism the connection \eqref{new_connection} can be presented as
    \begin{align}
        d+\frac{\hbar}{2} \Big(\frac{d(z_1-z_2)}{z_1-z_2}(E_{3,2}&E_{2,3}+E_{4,1}E_{1,4})+\\+\frac{d(z_1-z_3)}{z_1-z_3}(E_{4,2}E_{2,4}+E_{3,1}E_{1,3})+&\frac{d(z_2-z_3)}{z_2-z_3}(E_{4,3}E_{3,4}+E_{2,1}E_{1,2})\Big),
    \end{align}
    where $\overline{x}$ is a residue of an integer $x$ modulo $3$.
\end{ex}

Recall the definition of the boundary Casimir connection (see \cite{BK})
\begin{equation}
    \nabla_{{\rm bCas}}:=d-\frac{\hbar}{2}\sum_{a,b\in I,a\neq b} \frac{d(z_a-z_b)}{z_a-z_b} B_{a,b}^2,\quad I=\{-n,\dots,-1,1,\dots,n\}
\end{equation}
where $B_{a,b}:=E_{a,b}-E_{b,a}$ and assume that it acts on the trivial vector bundle with fiber $\Lambda^{\bullet}(V \otimes U)$ and the base space $\{(z_{-n},\dots, z_{-1},z_1,\cdots , z_{n}) \in \mathbb{C}^{2n} \mid z_i \neq z_j\}$ where $U$ is the tautological representation for $\mathfrak{so}(2n)$. Note that $\Lambda^{\bullet}W\otimes \Lambda^{\bullet}W$ contains the representation $\Lambda^{l}U$ of $\mathfrak{so}(2n)$ for any $l$. We can relate $\nabla_{{\rm bCas}}$ and \eqref{new_connection} as follows.\\

\begin{thm}
    Consider the pair of embeddings of an $O(V)\times\mathfrak{so}(U)$-module
    \begin{equation}
        \Lambda^{\bullet}(V^{\oplus 2} \otimes W) \overset{i_1}{\hookleftarrow} \Lambda^{l_1}U\otimes \dots \otimes \Lambda^{l_t}U \overset{i_2}{\hookrightarrow} \Lambda^{\bullet}(V \otimes U)
    \end{equation}
    for some non-negative integers $l_1,\dots,l_t$. The fiber restriction of $\nabla_{{\rm bCas}}$ with respect to the map $i_2$ is gauge equivalent to the residue connection
    \begin{equation}
        \nabla_{{\rm bCas}}':= d -\frac{\hbar}{2}\sum_{1\le a<b\le n}\frac{dz_a-dz_b}{z_a-z_b}(B_{a,b}^2+B_{a,-b}^2+B_{-a,b}^2+B_{-a,-b}^2)
    \end{equation}
    with $z_a=z_{-a}$ $\forall a$. The fiber restriction of \eqref{new_connection} with respect to the map $i_2$ is gauge equivalent to $\nabla_{{\rm bCas}}'$.
\end{thm}
\begin{proof}
    Consider an isomorphism of two $\mathfrak{so}(2n)$ algebras preserving nondegenerate diagonal and antidiagonal symmetric forms given by
\begin{equation}\label{out}
    e_j \mapsto \frac{e_j+e_{-j}}{\sqrt{2}}, e_{-j}\mapsto {\rm i}\frac{e_{-j}-e_j}{\sqrt{2}}.
\end{equation}
Since $GL(U)$ acts on $\Lambda^{l_1}U\otimes \dots \otimes \Lambda^{l_t}U$, the restriction of the dual connection \eqref{new_connection} transforms under this isomorphism into
    \begin{equation}\label{components}
        -\frac{1}{2}\left(B_{a,b}^2+B_{a,-b}^2+B_{-a,b}^2+B_{-a,-b}^2-2{\rm i}B_{b,-b}\right)
    \end{equation}
where $B_{a,b}:=E_{a,b}-E_{b,a}$. Here, we can remove the last term by an appropriate gauge transformation (as it comes from an element of the Cartan subalgebra of $\mathfrak{so}(2n)$ before the isomorphism).\\

Now, take the connection $\nabla_{{\rm bCas}}$ acting on $\Lambda^{\bullet}(V\otimes U)$ and substitute a change of variables $z_{-i}=\epsilon_i+z_i$. Then it can be rewritten as follows
\begin{align}
    &\nabla_{{\rm bCas}}=d-\frac{\hbar}{2}\sum_{1\le a<b\le n} \Big( B_{a,b}^2d\log(z_{a}-z_{b})+ B_{-a,b}^2d\log(z_{a}-z_{b}+\epsilon_a)+ B_{b,-b}^2 d\log(\epsilon_b)+\\&+ B_{a,-b}^2d\log(z_{a}-z_{b}-\epsilon_b)+ B_{-a,-b}^2d\log(z_{a}-z_{b}+\epsilon_a-\epsilon_b)+ B_{a,-a}^2d\log(\epsilon_a)\Big).
\end{align}

The gauge transformation of $\nabla_{{\rm bCas}}$ given by
\begin{equation}\label{gauge}
    \Psi=\exp\left(-\frac{\hbar(n-1)}{2}\sum_{a=1}^n\log(\epsilon_a)B_{a,-a}^2\right)
\end{equation}
removes all singular terms of $\nabla_{{\rm bCas}}$ with respect to $\epsilon_a$, so we can extend the conjugated connection to the space $\epsilon_a=0, z_a\neq z_b$ if $a\neq b$. Such restriction in the base clearly gives us the connection $\nabla_{{\rm bCas}}'$.
\end{proof}

We also have the analogues of \lemref{isomor} and \thmref{thm_isomor}.

\begin{lem}
    For all non-integer and large enough integer $t$ we have an isomorphism
    \begin{equation}\label{mult_space_o_t}
        \phi :\Hom_{\Rep(O_t)}(V_{[\lambda]}, V^{\otimes n }) \cong W_{\gamma^\top}[\beta],
    \end{equation}
    where $W_{\gamma^\top}$ is the irreducible $\mathfrak{so}_{2n}$-module of the highest weight $\gamma^\top$.
\end{lem}
\begin{proof}
    The proof of this lemma is similar to \lemref{isomor}, since we know that the joint highest weight vectors for $O_t\times SO_{2n}$ in $\Lambda^{\bullet}(V\otimes W)$ are necessarily joint highest weight vectors for $GL_t\times GL_m$ (\cite{How}[Section 4.3.5]). From this we see that these highest weight vectors and the relations on their descendants are independent of $t$, because we may embed a joint highest weight vector and its descendants into
    \begin{equation}\label{independence_O_t}
         \bigoplus_{i_1,\dots,i_{l(\lambda)}\ge 0} \Lambda^{\gamma_1-2i_1}W \otimes \dots \otimes \Lambda^{\gamma_{l(\lambda)} -2i_{l(\lambda)}}W \otimes \mathbb{C}\cdot 1  \otimes \mathbb{C}\cdot 1 \otimes \dots
    \end{equation}
    for a certain partition $\gamma$, but \eqref{independence_O_t} is still a finite-dimensional vector space. The choice of basis in $\Hom_{\Rep(O_t)}(V_{[\lambda]},V^{\otimes n})$ is repeated verbatim.
\end{proof}
\begin{thm}\label{KZ_connection_O_t_identification}
    The isomorphism $\phi$ identifies the connection \eqref{new_connection} on $W_{\gamma^\top}[\beta]$ and the KZ connection \eqref{KZ_O_t} on
    $\Hom_{\Rep(O_t)}(V_{[\lambda]}, V^{\otimes n })$. The flat sections of both connections are related as in \thmref{O_t_duality}.
\end{thm}

\begin{proof}
    Similar to \thmref{thm_isomor}.
\end{proof}

The usual methods of \cite{FMTV} for finding integral solutions can not be applied directly for $\nabla_{bCas}$ for two reasons. Firstly, the poles of $\nabla_{bCas}$ differ from the root hyperplanes of $\mathfrak{h}^*_{\mathfrak{so}_{2n}}$. Lastly, there does not seem to be a canonical lift of the connection formula to the level of universal enveloping algebra $U(\widetilde{\mathfrak{n}}_- \oplus \mathfrak{h}\oplus \widetilde{\mathfrak{n}}_+)$ where $\widetilde{\mathfrak{n}}_\pm$ are free Lie algebras generated by $f_i,e_i$ respectively.

\begin{prob}
    An interesting problem would be to find integral formulas for solutions to \eqref{new_connection} or $\nabla_{{\rm bCas}}$.
\end{prob}

\section{Drinfeld-Kohno theorem in Deligne categories.}

Let $t\in \mathbb{C}\setminus \mathbb{Z}$ and $\Rep(GL_t)$ be the corresponding symmetric Deligne category. Let $V_1,...,V_n,Y $ be arbitrary objects of $ \Rep(GL_t)$. Let $\hbar\in \Bbb C$, and consider the KZ equations  
\begin{equation}
    \frac{\partial F}{\partial z_i}=\hbar\sum_{j\ne i} \frac{\Omega_{ij}}{z_i-z_j}F
\end{equation}
where $\Omega_{ij}$ are the Casimir endomorphisms and 
$F$ is a locally holomorphic solution in variables $z_1,...,z_n$ (defined locally in some simply connected region of $\Bbb C^n$ where $z_i\ne z_j$ and extended analytically) with values in the space $\Hom_{\Rep(GL_t)}(Y,V_1\otimes...\otimes V_n)$. 
Let $\rho(\hbar)$ be the monodromy representation of the pure braid group $B_n$ 
(for some base point) defined by this equation, when without loss of generality $V_1=V_2=\dots=V_n=V$. Thus, $\rho(\hbar)$ 
is well defined up to an isomorphism. \\

On the other hand, assume that $\hbar\notin \Bbb Q$. Let $q=e^{\pi i\hbar}$ and consider the quantum Deligne category $\Rep_q(GL_t)$ -- the Skein category with parameters $q$ and $a:=q^t=e^{\pi i\hbar t}$. Furthermore, assume that $q,a$ are multiplicatively independent, so the category $\Rep_q(GL_t)$ is semisimple and naturally 
equivalent to $\Rep(GL_t)$ as an abelian category (i.e., simple objects of both categories are labeled by the same set - pairs of partitions). Let 
$X_q\in \Rep_q(GL_t)$ be the $q$-analog of the object $X\in \Rep(GL_t)$.
The category $\Rep_q(GL_t)$ is braided, so we have a braid group representation 
$\rho_q: B_n\to {\rm Aut} (\Hom_{\Rep_q(GL_t)}(Y_q,V_q^{\otimes n}))$. 

\begin{thm} The representations $\rho(\hbar)$ and $\rho_q$ are isomorphic. 
\end{thm} 

\begin{proof} Following Drinfeld (\cite{Dr1,Dr2}), define a new braided tensor structure on $\Rep(GL_t)$ using the KZ equations. Namely, consider the KZ equation for $n=3$ 
with $z_1=0,z_2=z,z_3=1$. We then get 
$$
 F'(z)=\hbar\left(\frac{\Omega_{12}}{z}+\frac{\Omega_{23}}{z-1}\right)F(z)
$$
where $F(z)\in \Hom_{\Rep(GL_t)}(Y,V_1\otimes V_2\otimes V_3)$. 
Define the associativity isomorphism 
$$
\Phi_{V_1V_2V_3}: (V_1\otimes V_2)\otimes V_3\to V_1\otimes (V_2\otimes V_3)
$$ 
as the suitably renormalized monodromy operator from $0$ to $1$ of this KZ equation  for arbitrary $Y$, and define the braiding by $\beta:=e^{\frac{\hbar \Omega}{2}}$. 
Denote this new braided tensor category by $\Rep(GL_t)(\hbar)$. 
In this category, the braid group action on $\Hom_{\Rep(GL_t)(\hbar)}(Y,V^{\otimes n})$ is given, up to isomorphism, by the monodromy of the KZ equation.\\

Thus we have two braided tensor structures on the same semisimple abelian category. We can further endow both categories with the ribbon structure by letting the balancing morphism act by $q^{-\langle \lambda+2\rho,\lambda\rangle}$ on simple objects of weight $\lambda$.
The category $\Rep_q(GL_t)$ has a universal property: braided tensor functors from $\Rep_q(GL_t)$ to a ribbon category $\mathcal C$ correspond to (rigid) objects $X$ in $\mathcal C$ of quantum dimension $[t]_q$, where
$[t]_q:=\frac{q^t-q^{-t}}{q-q^{-1}}$ such that the braiding $\beta_{XX}: X\otimes X\to X\otimes X$ satisfies the Hecke relation $(\beta-q)(\beta+q^{-1})=0$. 
So we have a braided tensor functor $\Rep_q(GL_t)\to \Rep(GL_t)(\hbar)$ 
sending the tautological object $X_q$ to $X$, which is clearly an equivalence. The Drinfeld-Kohno theorem follows. 
\end{proof} 

Consider analogous representation of the braid group 
\begin{equation}
    \rho_O(\hbar): B_n \rightarrow {\rm Aut}(\Hom_{\Rep (O_t)}(Y,V^{\otimes n}))
\end{equation}
where $Y\in \Rep(O_t)$ and $V$ is the defining object. And let
\begin{equation}
    \rho_{O,q}: B_n \rightarrow {\rm Aut}(\Hom_{\Rep (O_t)}(Y_q,V_q^{\otimes n}))
\end{equation}
be a similar representation for the BWM category. Then assuming $q$ and $q^t$ are multiplicatively independent, we have the following.
\begin{thm}
    The representations $\rho_O(\hbar)$ and $\rho_{O,q}$ are isomorphic.
\end{thm}
\begin{proof} The same as for $\Rep(GL_t)$. 
\end{proof} 
\appendix

\section{Hypergeometric solutions for $\lambda,\mu=0$ and $m=n=2$}
In this appendix we will describe explicit solutions of the KZ equations as in \eqref{OmegaSm} in terms of hypergeometric functions for the special case when $\lambda,\mu=0$ and $m=n=2$. In this case we are working in the space $\Hom_{\Rep(GL_t)}(\mathbbm{1}, V^{*\otimes 2}\otimes V^{\otimes 2}) \cong \CC[S_2]$, and the Casimirs $\Omega_{ij}$ act as in \eqref{OmegaSm}. Then letting $e, (12)$ be the two permutations in $\CC[S_2]$ we can express a KZ section as $\phi(z_1,z_2,z_3,z_4) = f(z_1,z_2,z_3,z_4) \cdot e + g(z_1,z_2,z_3,z_4) \cdot (12)$, and the KZ equations read
\begin{align}\label{KZS2}
    \begin{cases}
        \hbar^{-1} \partial_1 f &= \frac{g}{z_{12}} - \frac{tf+g}{z_{13}} \\
        \hbar^{-1} \partial_1 g &= \frac{f}{z_{12}} - \frac{f+tg}{z_{14}} \\
    \end{cases}
    \text{ and symm. eqs. for } z_i \mapsto z_{\pi(i)}, \ \pi \in \{(12)(34), (13)(24), (14)(23)\}
\end{align}
where for brevity we denote $z_{ij} := z_i - z_j$ and $\partial_i := \partial_{z_i}$. Also, denote $\Delta := t\hbar$. Then it is straightforward to check that the equations \eqref{KZS2} are solved by
\begin{align}
    f(z_1,z_2,z_3,z_4) &= \frac{A(z_1,z_2,z_3,z_4)}{z_{13}^{\Delta} z_{24}^{\Delta}} \\
    g(z_1,z_2,z_3,z_4) &= \frac{B(z_1,z_2,z_3,z_4)}{z_{14}^{\Delta} z_{23}^{\Delta}}.
\end{align}
where $A,B$ are functions depending on two parameters $c_1,c_2 \in \CC$:
\begin{align}
    A&:= c_1 \cdot\left(\frac{z_{14}z_{32}}{z_{12}{z_{34}}}\right)^{1-\Delta} \F \left( 1-\hbar-\Delta, 1+\hbar-\Delta; 2-\Delta; \frac{z_{14}z_{32}}{z_{12}{z_{34}}} \right) \\ &\qquad \qquad \qquad +c_2 \cdot \F \left( -\hbar, \hbar; \Delta; \frac{z_{14}z_{32}}{z_{12}{z_{34}}} \right) \\
    B&:= \hbar^{-1} z_{14}^{\Delta} z_{23}^{\Delta-1} z_{13}^{-\Delta+1} z_{24}^{-\Delta} z_{12} (\partial_1 A).
\end{align}
% ($c_1, c_2 \in \CC$ are arbitrary constants). 
Note that these solutions involve one integration (in the hypergeometric functions). For comparison, \thmref{duality} and \thmref{solthm} give us the formula for solutions in this case as explained below.\\

Let us fix two integration cycles $\Gamma_{l_i},i=1,2$ for $l^{-1}_1: (1,2,3,4) \to (2,3,1,4)$ and \linebreak $l^{-1}_2: (1,2,3,4) \to (2,1,4,3)$. Note that $f_{c(\sigma)}v$ is not zero only in the case if $c(\sigma(4))$ is $2$. One can see that the vectors $f_2f_1 f_3f_2v=f_2f_3 f_1f_2v$ both correspond to $e\in \CC[S_2]$. Analogously we have the following correspondence
\begin{align}
    f_1f_3f_2f_2v=f_3f_1 & f_2f_2v \sim 2e+2(12),\\
    f_1f_2f_3f_2v &\sim e+(12),\\
    f_3f_2f_1f_2v &\sim e+(12).
\end{align}
This allows us to compute $\omega$:
\begin{multline}
    \omega = \left(\frac{2t_1t_4 - (t_1+t_4)(t_2+t_3)+t_2^2 +t_3^2}{t_2t_3(t_4-t_2)(t_4-t_3)(t_1-t_2)(t_1-t_3)}e+\frac{2t_1t_4 - (t_1+t_4)(t_2+t_3)+2t_2t_3}{t_2t_3(t_4-t_2)(t_4-t_3)(t_1-t_2)(t_1-t_3)}(12)\right)\cdot \\ \cdot dt_1 \wedge dt_2 \wedge dt_3 \wedge dt_4.
\end{multline}
Then we'll have the solutions given by
\begin{multline}
    \int_{\Gamma_{l_i}} \exp{\left(\hbar( z_{12}t_1+z_{23}(t_2+t_3) + z_{34}t_4) \right)}\left( \frac{(t_2-t_3)^2}{(t_1-t_2)(t_1-t_3)(t_2-t_4)(t_3-t_4)} \right)^{-\hbar}\cdot \\ \cdot \omega \cdot z_{12}^{t-1} z_{13}z_{14}z_{23}z_{24}z_{34}
\end{multline}

Data sharing not applicable to this article as no datasets were generated or analysed during the current study. All authors declare that they have no conflicts of interest.

\end{document}